\documentclass[12pt]{amsart}

\usepackage{mathrsfs}
\usepackage{amsmath,amssymb,enumerate,tikz}
\usepackage{amsfonts,amssymb}
\usepackage{pinlabel}
\usepackage{latexsym,amsfonts,amssymb,verbatim,mathrsfs,amsthm}
\usepackage{amsmath,amsthm,amssymb,latexsym,graphics,textcomp,graphicx,subfigure}
\usepackage{eucal,eufrak}
\usepackage{colonequals}
\usepackage{graphicx}
\usepackage{url}
\usepackage{hyperref}

\theoremstyle{plain}
\newtheorem{theorem}{Theorem}[section]
\newtheorem{lemma}[theorem]{Lemma}
\newtheorem{corollary}[theorem]{Corollary}

\newtheorem{maintheorem}{Theorem}

\newtheorem{definition}[theorem]{Definition}

\newtheorem{proposition}[theorem]{Proposition}
\newtheorem{remark}[theorem]{Remark}

\newtheorem{claim}[theorem]{Claim}

\theoremstyle{definition}

\newtheorem{example}{Example}

\newcommand{\C}{\mathbb{C}}

\newcommand{\R}{\mathbb R}
\newcommand{\Z}{\mathbb Z}

\newcommand{\N}{\mathbb{N}}
\newcommand{\Hc}{\mathcal{H}}
\newcommand{\Lc}{\mathcal{L}}

\newcommand{\Hyp}{\mathcal{H}^{\rm hyp}}

\newcommand{\gm}{\gamma}

\newcommand{\omg}{\omega}

\newcommand{\Aff}{\textup{Aff}}
\newcommand{\SL}{\textup{SL}}
\newcommand{\GL}{\textup{GL}}
\newcommand{\ol}{\overline}

\newcommand{\ra}{\rightarrow}

\newcommand{\inter}{\mathrm{int}}

\renewcommand\mod{\text{ mod }}

\newcommand{\PP}{\mathbf{P}}
\newcommand{\RR}{\mathbf{R}}
\newcommand{\Sph}{\mathbb{S}}

\title[Existence of closed geodesics through a regular point]{Existence of closed geodesics through a regular point on translation surfaces}

\date{\today}
\author{Duc-Manh Nguyen \and  Huiping Pan \and  Weixu Su}

\address{
Duc-Manh Nguyen\newline
IMB Bordeaux, CNRS  UMR 5251 \newline
Universit\'e de Bordeaux,\newline 351, Cours de la Lib\'eration, 33405 Talence Cedex, France
}
\email{duc-manh.nguyen@math.u-bordeaux.fr}

\address{Huiping Pan\newline
Department of mathematics, Jinan University,\newline 510632, Tianhe, Guangzhou, P. R. China}
\email{panhp@jnu.edu.cn}

\address{Weixu Su \newline School of Mathematical Science, Fudan University,\newline 200433, Shanghai, P. R. China}
\email{suwx@fudan.edu.cn}

\begin{document}

\begin{abstract}
We show that on any translation surface, if a regular point is contained in some simple closed geodesic, then it is contained in infinitely many simple closed geodesics, whose directions are dense in $\mathbb{RP}^1$. Moreover, the set of points that are not contained in any simple closed geodesic is finite.
We also construct explicit examples showing that such points exist on some translation surfaces.
For a surface in any hyperelliptic component, we show that this finite exceptional set is actually empty.
The proofs of our results use Apisa's classifications of periodic points and of $\GL(2,\R)$ orbit closures in hyperelliptic components, as well as a recent result of Eskin-Filip-Wright.
\end{abstract}
\maketitle



\vskip 20pt
\section{Introduction}
Translation surfaces are flat surfaces with conical singularities such that the holonomy of any closed curve (not passing through the singularities) is a translation of $\R^2$.
On such a surface, a natural question one may ask is whether there exists any simple closed geodesic that avoids singularities and passes through a given regular point.
This question is particularly relevant in view of applications to the theory of billiards.

Throughout this paper, the translation surfaces we consider will be closed, and  by \emph{closed geodesics} we will mean those that do not pass through singularities.
In the literature, Masur \cite{Mas} proved that the set of directions in  which there exist simple closed geodesics  is a dense subset of $\mathbb{RP}^1\simeq \R/\pi\Z$. Later, Bosheritzan-Galperin-Kr\"uger-Troubetzkoy \cite{BGKT} improved Masur's result, they proved that the set of tangent vectors that generate closed geodesics is dense in the unit tangent bundle of the surface (see also \cite{MT}). Vorobets \cite{Vor1,Vor2} proved that for any translation surface, almost every regular point on it is contained in infinitely many simple closed geodesics, whose directions are dense in $\mathbb{RP}^1$.

In this paper, we will show

\begin{maintheorem}\label{thm:finite}
 On any translation surface, every regular point, except for finitely many,  is contained in some simple closed geodesic.
\end{maintheorem}

Furthermore, we have

\begin{maintheorem}\label{thm:gen:dichotomy}
   Let $(X,\omega)$ be a translation surface. If a regular point on $(X,\omega)$ is contained in some simple closed geodesic, then
  \begin{enumerate}[(i)]
  \item it is contained in infinitely many simple closed geodesics $\{\gamma_n\}_{n\geq1}$, whose directions are dense { in $\mathbb{RP}^1$};
  \item the union $\cup_{n\geq1}\gamma_n$ is dense in $(X,\omega)$.
\end{enumerate}
\end{maintheorem}

\begin{remark}
 On any translation surface, a regular point being contained in a simple closed geodesic  is equivalent to it being contained in the interior  of some cylinder.
\end{remark}

For translation surfaces in the hyperelliptic components, we show that the finite exceptional set mentioned in Theorem~\ref{thm:finite} is actually empty.

\begin{maintheorem}\label{thm:hyp}
Let $(X,\omega)$ be a translation surface in one of the hyperelliptic components $\mathcal{H}^{\rm hyp}(\kappa)$ with  $\kappa \in \{(2g-2),(g-1,g-1)\}$ for $g\geq 2$.
Then any regular point of $(X,\omega)$ is contained in infinitely many simple closed geodesics.
\end{maintheorem}

It is not difficult to construct translation surfaces { on which there exist} regular points not contained in any simple closed geodesic (see Section~\ref{sec:nocyl:pt}).
Thus the exceptional subset is not empty in general.

\subsection*{Outline} Here below we will give an outline of the proofs.
All the definitions and necessary materials will be introduced in Section~\ref{sec:background}.
It is worth noticing that, even though the statements of these theorems only concern one individual surface, their proofs use in an essential way the structure of its $\GL(2,\R)$-orbit closure.

\begin{enumerate}
\item The strategy to prove Theorem~\ref{thm:finite} is to show that any regular point in a translation  surface $(X,\omega)$ that is not contained in any simple closed geodesic is a ``periodic point'' in the sense of Apisa  \cite{Apisa_generic}.

\medskip

\item  Theorem~\ref{thm:gen:dichotomy} is based on a result of Eskin-Mirzakhani-Mohammadi (see Theorem~\ref{thm:equid}) which states that, given a translation surface { with marked points $(X,\omega,\{p_1,\dots,p_n\})$},
    for any interval $I=[a,b]\subset \mathbb R/2\pi$ with $b\neq a$, the sector $\{a_t r_\theta (X,\omega,\{p_1,\dots,p_n\}): t>0, \theta\in I\}$    is equidistributed in the $\SL(2,\R)$ orbit closure of $(X,\omega,\{p_1,\dots,p_n\})$,
where $a_t=\left(
      \begin{smallmatrix}  e^t&0\\  0&e^{-t}  \end{smallmatrix}
       \right)$ and
       $ r_\theta=\left(
      \begin{smallmatrix} \cos\theta&-\sin\theta\\  \sin\theta&\cos\theta  \end{smallmatrix}
       \right)$.

 \medskip

 \item For Theorem~\ref{thm:hyp}, we now assume that $(X,\omega)$ belongs to some hyperelliptic component  $\Hc^{\rm hyp}(\kappa)$, with $\kappa \in \{(2g-2),(g-1,g-1)\}$.
By Theorem~\ref{thm:gen:dichotomy} and Theorem~\ref{thm:no:cyl:pt:periodic}, we only need to show that every periodic point $p$ in $(X,\omega)$ is actually contained in at least one simple closed geodesic.
Note that by a result of Apisa~\cite{Apisa_hyp,Apisa_hyprk1}, we have a complete list of all the possibilities for $\Lc=\ol{\GL(2,\R)\cdot (X,\omega)}$.

If $\Lc=\Hc^{\rm hyp}(\kappa)$,  then by \cite[Theorem 1.5]{Apisa_generic}, $p$ must be a regular Weierstrass point of $X$. It follows from Proposition~\ref{prop:inv:simp:cyl} that $p$ is contained  in a simple cylinder of $(X,\omega)$. Thus the theorem is proved for this case. By a simple observation (cf. Lemma~\ref{lemma:bc}), this also yields the proof for the case that $\Lc$ consists of translation covers of surfaces in another hyperelliptic component of lower genus. See Definition \ref{def:trans_cover} for the notion of translation cover.

If $\Lc$ is a non-arithmetic eigenform locus in genus two then either $p$ is Weierstrass point, or the $\GL(2,\R)$ orbit closure of $(X,\omega,p)$ contains a surface with one marked point $(Y,\eta,q)$ where $(Y,\eta)$ is a Veech surface and $q$ is a regular Weierstrass point of $Y$. Note that the latter case occurs only if $\Lc$ is the golden eigenform locus (see \cite{Apisa_periodic_g2}). In both cases, Proposition~\ref{prop:inv:simp:cyl} allows us to conclude.
By Lemma~\ref{lemma:bc}, this also yields the proof for the case $\Lc$ arises from a non-arithmetic eigenform locus by a covering construction.

Assume now that $\Lc$ is a closed orbit, which means that $(X,\omega)$ is a Veech surface. In this case the theorem follows from  Proposition~\ref{prop:veechsf}.

Finally, assume that $\Lc$ consists of translation covers (see Definition \ref{def:trans_cover}) of flat tori branched over two points. In this case, we will show that the $\GL(2,\R)$-orbit closure of $(X,\omega,p)$ contains a surface with one marked point $(Y,\eta,q)$, where $(Y,\eta)$ is a Veech surface and $q \in Y$ is a regular point. We then use Proposition~\ref{prop:veechsf} and Lemma~\ref{lm:finite:cyl:closed} to conclude.
\end{enumerate}

\subsection*{Acknowledgements:}
The first-named author thanks Fudan University of Shanghai for its hospitality which helped to initiate the collaboration resulting in this work. The second author would like to thank IMB Bordeaux for the hospitality during his visit. He also wants to thank Alex Wright for explaining his joint work with Paul Apisa. We are grateful to the referee for the careful
reading of the manuscript, with many corrections and useful comments.
{H. Pan and W. Su are partially supported by NSFC No: 11671092 and No: 11631010.}

\section{Background}\label{sec:background}
In this section, we recall the basic definitions and important results that are used in the proofs of the main theorems.

\subsection{Stratum} A translation surface can be defined as a pair $(X,\omega)$ where $X$ is a compact Riemann surface of genus $g\geq1$ and $\omega$ is a holomorphic one-form on $X$.
In this description, the singularities of the flat metric correspond to the zeros of the one-form.
{ We will also consider translation surfaces with some  regular marked points. In this situation, those regular points will be considered as zeros of order $0$. Given an integer $g\geq 2$, let $\kappa=(k_1,\dots,k_n)$ be an integral vector where $k_i\geq 0, \; i=1,\dots,n$, and $k_1+\dots+k_n=2g-2$. Then the moduli space of translation surfaces $(X,\omega)$ where $\omega$ has $n$ zeros with orders given by $\kappa$ will be denoted by $\mathcal{H}(\kappa)$. The space $\mathcal{H}(\kappa)$ will be called a {\em stratum}.}

It is a well-known fact that $\mathcal{H}(\kappa)$ is a complex orbifold of dimension $2g+n-1$.
For any $(X,\omega) \in \Hc(\kappa)$, we choose a basis $(\gamma_1,\dots,\gamma_{2g+n-1})$ of the relative homology group $H_1(X,\Sigma; \Z)$, where $\Sigma$ is  set of zeros of $\omega$ (including the ones of order $0$).
By a slight abuse of notation, for $(X',\omega')\in \Hc(\kappa)$ close to $(X,\omega)$,  we will also denote by $\gamma_i$ the corresponding element  of $H_1(X',\Sigma';\Z)$, where $\Sigma'$ is the zero set of $\omega'$. Then the map
\begin{equation*}
  \Phi: (X,\omega) \mapsto (\int_{\gamma_1}\omega,\dots,\int_{\gamma_{2g+n-1}}\omega) \in \C^{2g+n-1}
\end{equation*}
is a  homeomorphism from a neighborhood of $(X,\omega)$ onto its image, thus defines a local chart for $\Hc(\kappa)$ near $(X,\omega)$. The map $\Phi$ is called the {\em period mapping}, and the associated local coordinates are called {\em period coordinates}. For a thorough introduction to the subject, we refer to \cite{MT, Zorich:survey}.

\subsection{Saddle connection and cylinder} A {\em saddle connection} of $(X,\omega)$ is a geodesic segment for the flat metric defined by $\omega$ ending at the zeros of $\omega$ which does not contain any zero in its interior.
A {\em  cylinder} of $(X,\omega)$ is an {\em open} subset of $X$ which is isometric to $(\R/c\Z)\times(0,h)$, with $c,h \in \R_{>0}$, and not properly contained in another subset with the same property.
The parameters $c$ and $h$ are called the {\em circumference} and the {\em height} of the cylinder respectively.
The isometric mapping from $(\R/c\Z)\times(0,h)$ to $(X,\omega)$ can be extended to a map from $(\R/c\Z)\times[0,h]$ to $X$.
The images of $(\R/c\Z)\times\{0\}$ and $(\R/c\Z)\times \{h\}$ under this map are called the {\em boundaries} or {\em borders} of the cylinder.
Each boundary component is a concatenation of some saddle connections in the same direction.
Note that the two boundary components are not necessarily disjoint as subsets of $X$.
A cylinder  is called  {\em simple} if each of its boundary components consists of a single saddle connection.
Note that every simple closed geodesic is contained in a unique  cylinder, and a  cylinder is foliated by  parallel simple closed geodesics.

\subsection{$\GL(2,\R)$-action}
There is a natural $\GL(2,\R)$ action on $\mathcal{H}(k)$, which acts on $(X,\omega)\in \mathcal{H}(\kappa)$ by post-composition with the atlas maps for $(X,\omega)$.
More precisely, let $\{(U_i,\phi_i)\}_{i\in I}$ be an atlas of charts for $(X,\omega)$ which defines the translation surface structure, covering $X$ except the zeros of $\omega$ (see \S 1.8 of \cite{MT}).
Then for $a\in\GL(2,\R)$, $a\cdot(X,\omega)$  is defined by the atlas  $\{(U_i,a\circ\phi_i)\}_{i\in I}$.
{In period coordinates, the action of $a$ is given by applying $a$ to each  coordinate (here, we use the standard identification $\C\simeq\R^2$).}

There is a deep connection between the dynamics of $\GL(2,\R)$ on $\Hc(\kappa)$ and the dynamics of straight line flows on individual  translation surfaces. In more concrete terms, the geometric and dynamical properties of a translation surface is often encoded in its $\GL(2,\R)$-orbit closure in moduli space.

We recall the recent breakthroughs of Eskin-Mirzakhani (\cite{EM}) and Eskin-Mirzakhani-Mohammadi (\cite{EMM}).
Define an {\em affine invariant submanifold} $\mathcal M$ in a stratum  to be an immersed submanifold which is defined locally by homogeneous linear equations with real coefficients in the period coordinates (see \cite{EMM} and \cite{Wright1} for more details).

\begin{theorem}[\cite{EMM}]\label{thm:EMM}
The $\GL(2,\R)$ orbit closure of any translation surface is an affine invariant submanifold in its stratum. Any $\GL(2,\R)$ invariant, closed subset of a stratum is a finite union of affine invariant submanifolds.
\end{theorem}

Note that each stratum is itself a $\GL(2,\R)$-orbit closure, by ergodicity of  Teichm\"uller geodesic flow (\cite{Mas_ergodic}). A surface whose $\GL(2,\R)$ orbit is dense in its stratum is called {\em generic}.
The above result of Eskin-Mirzakhani-Mohammadi is followed by further results by Avila-Eskin-M\"{o}ller~\cite{AEM12}, Wright~\cite{Wright1,Wright2}, Filip~\cite{Fil} that provide more information about the orbit closures. However, obtaining the complete classification of $\GL(2,\R)$-orbit closures for each stratum remains a major challenge in Teichm\"uller dynamics.

Let $a_t=\left(
      \begin{smallmatrix}  e^t&0\\  0&e^{-t}  \end{smallmatrix}
       \right)$ and
       $ r_\theta=\left(
      \begin{smallmatrix} \cos\theta&-\sin\theta\\  \sin\theta&\cos\theta  \end{smallmatrix}
       \right)$. Let $\mathcal H_1(\kappa)\subset \mathcal{H}(\kappa)$ be the set of translation surfaces of area one.
       The following theorem describes the equidistribution of sectors.
\begin{theorem}[Theorem 2.6 in \cite{EMM}]\label{thm:equid}
  Let $(X,\omega)\in \mathcal H_1(\kappa)$ be a translation surface of are one. Let $\mathcal L\subset\mathcal H_1(\kappa)$ be the $\SL(2,\R)$ orbit closure of $(X,\omega)$. Then for any interval $I=[a,b]\subset \mathbb R/2\pi$ with $b\neq a$, the sector $\{a_t r_\theta \cdot (X,\omega): t>0,\theta\in I
       \}$  is equidistributed in $\mathcal L$.
\end{theorem}

{\remark\label{rmk:equid} Theorem~\ref{thm:equid} also holds for  translation surfaces with any number of regular marked points.} 

\subsection{Veech surfaces}

For a translation surface $(X,\omega)$, let $\Aff(X,\omega)$ be the group of orientation-preserving self-homeomorphisms of $X$ which are given by affine maps in the local charts defined by $\omega$. Elements of $\Aff(X,\omega)$ are called {\em affine automorphisms} of $(X,\omega)$ (see~\cite{GJ,KS00,MT}).
To each affine automorphism $f\in \Aff(X,\omega)$, one can associate a matrix $a(f)\in \SL(2,\R)$ which is the derivative of $f$ in the local charts defined by $\omega$, i.e. the linear part of $f$. The image of $\Aff(X,\omega)$ in $\SL(2,\R)$ under the map $f\mapsto a(f)$ is called the \emph{Veech group} of $(X,\omega)$, denoted by $\SL(X,\omega)$. The group $\SL(X,\omega)$ is also the stabilizer of $(X,\omega)$ under the action of $\SL(2,\R)$.
If $\SL(X,\omega)$ is a lattice of $\SL(2,\R)$, i.e. $\SL(2,\R)/\SL(X,\omega)$ has finite volume with respect to the Haar measure on $\SL(2,\R)$, then $(X,\omega)$ is called a {\em Veech surface} (or {\em lattice surface}).

\begin{theorem}[\cite{SW1,Veech2}]\label{thm:veech-closed}
A translation surface is a Veech surface if and only if its $\GL(2,\R)$ orbit is closed.
\end{theorem}

Since the parallel transport on a translation surface does not change the directions of the tangent vectors, for any direction $\theta \in \mathbb{RP}^1$, we have a (singular) foliation of the surface whose leaves are geodesics in this direction.
This foliation is said to be {\em uniquely ergodic} if each of its leaves is dense and there is a unique transverse measure up to a multiplicative constant.
On the other hand, this foliation is said to be {\em periodic} if each of its leaves is either a saddle connection or a closed geodesic. 
Veech surfaces are of particular interest because they have more ``symmetries" than  other surfaces. Moreover, they have optimal dynamical behaviors.
Namely, we have
\begin{theorem}[\cite{Veech2}  Veech dichotomy]\label{thm:Veech_dich}
Let $(X,\omega)$ be a Veech surface, then for every direction, the directional flow is either periodic or uniquely ergodic.
\end{theorem}

As a consequence, we get
\begin{corollary}\label{cor:veech:periodic}
On a Veech surface, the direction of any saddle connection is  periodic.
\end{corollary}

\subsection{Hyperelliptic components}\label{sec:backgrd:hyp:comp}
By  Kontsevich-Zorich's result (\cite{KZ}), each stratum $\Hc(\kappa)$ has at most three connected components.
If $\kappa=(2g-2)$ or $(g-1,g-1)$, then $\mathcal{H}(\kappa)$ contains a special  called {\em hyperelliptic component}, which will be denoted by $\mathcal{H}^{\rm hyp}(\kappa)$. Let us describe $\mathcal{H}^{\rm hyp}(\kappa)$ in details.

For $g\geq 2$,  let $\mathcal{Q}(-4+\ell,-1^\ell)$ be the set of meromorphic quadratic differentials on the Riemann sphere with one zero of order $-4+\ell$ and $\ell$ simple poles, where $\ell=2g+1$ if $\kappa=(2g-2)$, or $2g+2$ if $\kappa=(g-1,g-1)$.
Then every translation surface in $\Hc^{\rm hyp}(\kappa)$ is the holonomy double cover of a quadratic differential in $\mathcal{Q}(-4+\ell, -1^\ell)$.
The involution of $(X,\omega)$ induced by the double cover is the {\em hyperelliptic involution} of $X$.

For $g=1$, by convention, we consider a flat torus with one or two marked points as a hyperelliptic translation surface and denote the corresponding strata by $\Hc(0)$ and $\Hc(0,0)$ respectively. Surfaces in $\Hc(0)$ (resp. in $\Hc(0,0)$) are {holonomy} double covers of quadratic differentials in $\mathcal{Q}(-1^4)$ (resp.  $\mathcal{Q}(0,-1^4)$).

All Riemann surfaces of genus two are hyperelliptic, and the strata $\Hc(2)$ and $\Hc(1,1)$  only contain hyperelliptic components.
{The following classification theorem of $\GL(2,\R)$-orbit closures of translation surfaces in genus two was obtained by McMullen, prior to the works of Eskin-Mirzakhani and Eskin-Mirzakhani-Mohammadi (parts of this classification were also obtained by Calta~\cite{Cal}).
In this paper, we will denote the complex dimension of an orbit closure $\mathcal L$ by $\dim \mathcal{L}$.
All dimensions we mention are complex dimensions.

\begin{theorem}[McMullen~\cite{McM1,McM2,McM_spin}]
Let $(X,\omega)$ be a translation surface in genus two, and let $\Lc$ be the closure of its $\GL(2,\R)$ orbit in the corresponding stratum.
If $(X,\omega)\in \Hc(2)$, then either
\medskip

\begin{itemize}
\item[(1-a)] $\Lc=\Hc(2)$, or

\medskip

\item[(1-b)] $\Lc=\GL(2,\R)\cdot(X,\omega)$;  in this case, the Jacobian $\textup{Jac}(X)$ of $X$ admits  real multiplication by a quadratic order of discriminant $D$ with $\omega$ as an eigenform.
\end{itemize}

\medskip

If $(X,\omega) \in \Hc(1,1)$, then one of the following possibilities holds:

\medskip
\begin{itemize}
  \item[(2-a)] $\Lc=\Hc(1,1)$.
\medskip
  \item[(2-b)] $\Lc$ is contained in the locus of pairs $(X,\omega)$ such that $\textup{Jac}(X)$ admits  real multiplication by a quadratic order of discriminant $D$ with $\omega$ as an eigenform, and $\dim \Lc=3$. 
   \medskip
  \item[(2-c)] $\Lc=\GL(2,\R)\cdot(X,\omega)$; in this case $\Lc$ must be contained in an orbit closure of the type described in  (2-b).
\end{itemize}
\end{theorem}
We refer to \cite{McM1} for the precise definitions.
Orbit closures of the type (2-b) are called {\em eigenform loci} in $\Hc(1,1)$ and indexed by the discriminant $D$.
If $D=d^2$ for some integer $d$, then the locus is called {\em arithmetic}, otherwise it is called {\em non-arithmetic}.
In~\cite{McM2}, McMullen showed that a non-arithmetic eigenform locus of type (2-b) and  discriminant $D\neq 5$ does not contain any closed $\GL(2,\R)$ orbit.
For $D=5$, the corresponding
eigenform locus, which is called the {\em golden eigenform locus}, contains a unique closed  $\GL(2,\R)$ orbit.

\begin{definition}[Translation cover]\label{def:trans_cover}
Let  $(X_1,\omega_1),(X_2,\omega_2)$ be two translation surfaces.
If there exists a  (branched) cover $\pi: X_1 \to X_2$ such that $\omega_1 =\pi^*\omega_2$, then $(X_1,\omega_1)$ is called a {\em translation cover} of $(X_2,\omega_2)$. In the case where $(X_2,\omega_2)$ is a flat torus, $(X_1,\omega_1)$ is called a {\em torus cover}.
\end{definition}
Whenever we mention a (branched) cover in this paper, we mean a translation cover.

Recently, using the results of \cite{EM} and \cite{EMM} together with technical tools developed in \cite{MW1,MW2}, Apisa obtained the following classification of $\GL(2,\R)$ orbit closures in the hyperelliptic components of any genus.
\begin{theorem}[\cite{Apisa_hyp,Apisa_hyprk1}]\label{thm:Apisa}
Let $(X,\omega)$ be a translation surface in some hyperelliptic component $\Hyp(2g-2)$ or $\Hyp(g-1,g-1)$, and $\mathcal L$ the corresponding $\GL(2,\R)$ orbit closure. Then $\mathcal L$ is one of the followings:

\medskip

\begin{enumerate}
 \item If $\dim \mathcal L=2$, then $\mathcal L$ is a closed orbit, i.e. $(X,\omega)$ is a Veech surface.
\medskip
 \item If $\dim \mathcal L=3$ and $\mathcal L$ is arithmetic, { that is $\mathcal L$ is locally defined by linear equations with rational coefficients in period coordinates}, then $\mathcal L$ is a branched covering construction over  $\mathcal H(0,0)$.
     \medskip
 \item If $\dim \mathcal L=3$ and $\mathcal L$ is nonarithmetic, then $\mathcal L$ is a branched covering construction over some non-arithmetic eigenform locus in $\mathcal H(1,1)$.
     \medskip
 \item If $\dim \mathcal L=2r$ with $1<r<g$, then $\mathcal L$ is a branched covering construction over $\mathcal H^{hyp}(2r-2)$.
    \medskip
 \item If $\dim \mathcal L=2r+1$ with $1<r<g$, then $\mathcal L$ is a branched covering construction over  $\mathcal H^{hyp}(r-1,r-1)$.
     \medskip
 \item $\mathcal L$ is the whole stratum.
 \end{enumerate}
 \bigskip
  The covers are branched over zeros of the holomorphic one-forms and commute with the hyperelliptic involution.
\end{theorem}

\begin{remark}\hfill
  \begin{itemize}
    \item In genus three, this classification was obtained by Nguyen-Wright~\cite{NW} and Aulicino-Nguyen~\cite{AN}.
    \item The classification of Veech surfaces (Teichm\"uller curves) remains open, even in genus two.
 \end{itemize}
\end{remark}

\subsection{Marked points and periodic points}
A point $p$ on a translation surface $(X,\omega)$ is called \emph{regular} if it is not a zero of $\omg$.
A regular point $p$ on $(X,\omg)$ is called a \textit{periodic point} if the orbit closures $\overline{\GL(2,\R)\cdot(X,\omega,p)}$ in  $\mathcal{H}(\kappa,0)$ and $\overline{\GL(2,\R)\cdot(X,\omega)}$ in $\mathcal{H}(\kappa)$ have the same dimension (\cite{Apisa_generic, Apisa_periodic_g2,AW, GHS, Mol1}). For a Veech surface, this is equivalent to saying that the orbit of $p$ under $\Aff(X,\omega)$ is finite.

When $(X,\omega)$ is a surface in some hyperelliptic component, we have the following result for the periodic points of $(X,\omega)$, due to Apisa and M\"{o}ller.
\begin{theorem}[\cite{Apisa_generic,Apisa_periodic_g2,Mol1}]\label{thm:pt}
Let $(X,\omega)$ be a primitive translation surface in some hyperelliptic component $\mathcal{H}^{hyp}(\kappa)$, where $\kappa=(2g-2)$ or $(g-1,g-1)$.
\begin{enumerate}
\item If $\overline{\GL(2,\R)\cdot (X,\omega)}=\mathcal{H}^{hyp}(\kappa)$, then the set of periodic points of $(X,\omega)$ are the set of regular Weierstrass points.
   \medskip
\item If $(X,\omega)$ is a primitive Veech surface of genus two, then the set of periodic points of $(X,\omega)$ are the set of regular Weierstrass points.
   \medskip
\item If $(X,\omega)\in \mathcal{H}^{hyp}(1,1)$ contains a periodic point $p$ which is not a Weierstrass point, then $\overline{\GL(2,\R)\cdot (X,\omega)}$ is the golden eigenform locus. 
\end{enumerate}
\end{theorem}

Let $(X,\omega)$ be a generic surface in the golden eigenform locus (cf. Section~\ref{sec:backgrd:hyp:comp}), and $p\in(X,\omega)$ be a periodic point which is not a Weierstrass point. Then it is shown in \cite[\S2]{Apisa_periodic_g2} (see also \cite{KM}) that the $\GL(2,\R)$ orbit closure of $(X,\omega,p)$ in $\mathcal {H}(1,1,0)$ contains a translation surface with one marked point $(Y,\eta,q)$, where $(Y,\eta)$ is a Veech surface and $q$ is a regular Weierstrass point.

\section{Existence of regular points not contained in any simple closed geodesic}\label{sec:nocyl:pt}
In this section, we will give a family of translation surfaces which have regular points not contained in any simple closed geodesic.

\begin{proposition}\label{prop:double}
Let $(Y,q)$ be a meromorphic quadratic differential that has exactly one simple pole (the other singularities are  zeros). Let $\pi:(X,\omega^2)\to(Y,q)$ be the holonomy double cover.
Then the pre-image of the simple pole under the double cover is not contained in any simple closed geodesic.
\end{proposition}\label{prop:unique:sim:pole}
\begin{proof}
Let $y\in Y$ be the unique simple pole, and $\tilde{y}\in X$ be its preimage under the double cover $\pi$. Let $i$ be the involution of $X$ that permutes the the sheets of the cover $\pi$. It is clear that $\tilde y$ is the only regular point which is invariant under $i$. Suppose to a contradiction  that $\tilde{y}$ is contained in some simple closed geodesic, say $\gamma$. Then $\gamma$ is invariant under the involution $i$. Since orientation-reversing isometries of the circle have exactly two fixed points,  there exists another fixed point of $i$ which is contained in $\gamma$. This contradicts the fact that $\tilde{y}$ is the only non-singular fixed point of $i$.
\end{proof}

\begin{remark}\label{rk:unique:sim:pole}\hfill
\begin{itemize}
\item[(i)] With the above notation, if  $(Y,q)$ has the property that every regular point is contained in a simple closed geodesic, then $\tilde{y}=\pi^{-1}(y)$  is the unique regular point of $(X,\omega)$ that is not contained in any simple closed geodesic.

\item[(ii)] { Let $\mathcal{Q}$ be the stratum of $(Y,q)$.  Combined with Theorem~\ref{thm:no:cyl:pt:periodic} below and \cite[Th. 1.4]{AW}, Proposition~\ref{prop:unique:sim:pole} implies that if $\mathcal{Q}$ does not consist entirely of hyperelliptic surfaces, and $(Y,q)$ is  generic in $\mathcal{Q}$, then the preimage of the simple pole of $q$ is the unique regular point in $X$ which is not contained in any simple closed geodesic.}
\end{itemize}
\end{remark}

To close this section, we construct  an explicit  surface, { which is square-tiled}, on which there is a unique regular point that is not contained in any simple closed geodesic.

\begin{figure}[htb]
    \centering
    \includegraphics[width=10cm]{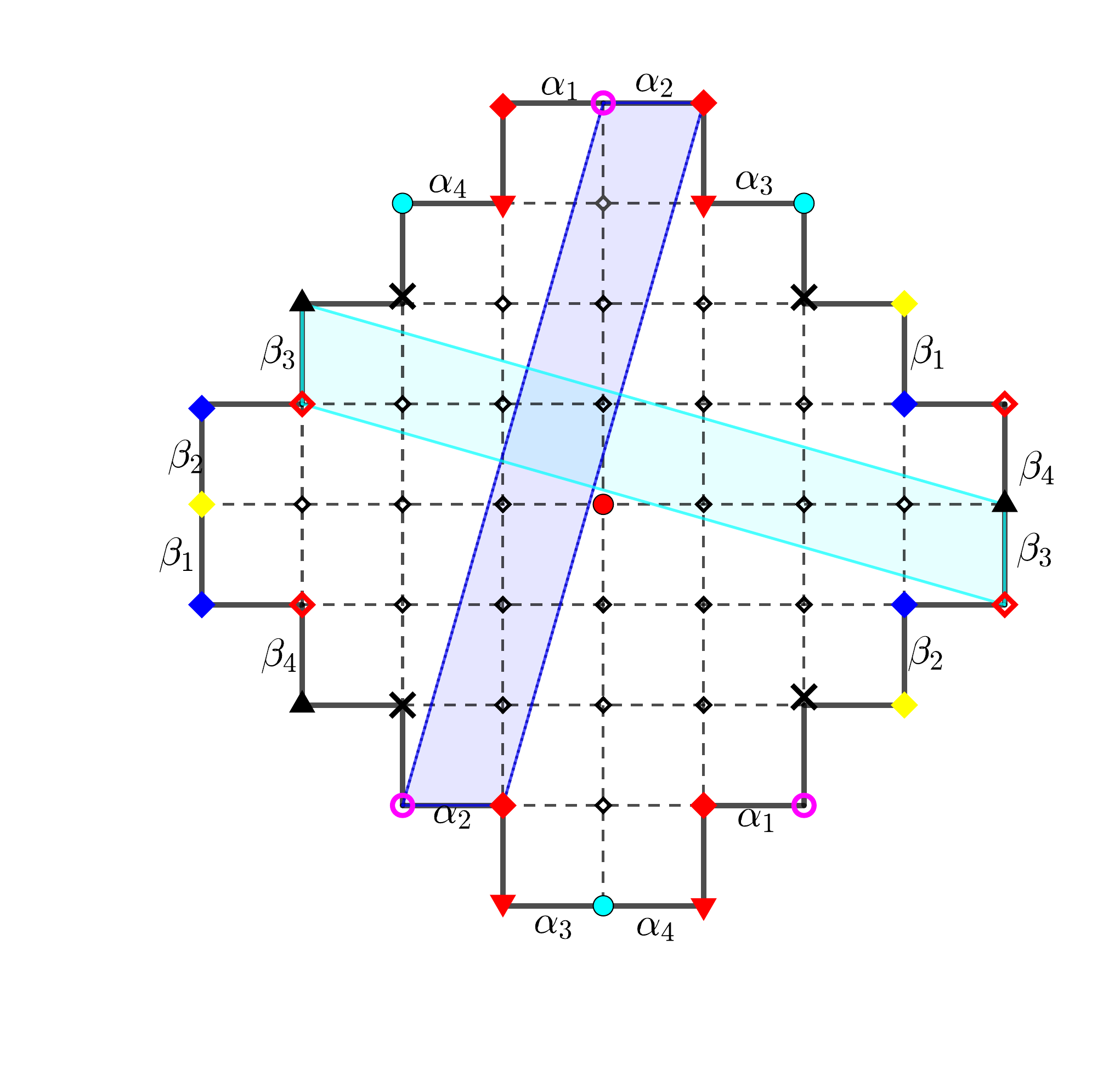}
    \caption{A translation surface $(X,\omg)\in\mathcal H(1,1,1,1,2)$. Each labeled edge is identified with the edge of the same label, each unlabeled horizontal (resp. vertical) edge is identified with the horizontal edge in the same vertical strip (resp. vertical edge in the same horizontal strip). }
    \label{Fig:H11112}
\end{figure}

\begin{example}
Consider the example $(X,\omg)$ in Figure \ref{Fig:H11112}. It is obtained from a central symmetric polygon $\mathcal P$ via edges identification. It is a translation surface in $\mathcal H(1,1,1,1,2)$,  tiled by 40 squares. It is also a branched double cover over a quadratic differential $(Y,q)\in \mathcal{Q}(1,2,2,-1) $.

Let $O\in(X,\omg)$ be the point represented by the center of the polygon which is the preimage of the simple pole of $q$.  By Proposition \ref{prop:double}, $O$ is a not contained in any simple closed geodesic of $(X,\omega)$.
Moreover, in this case, any other regular point is contained in some simple closed geodesic of $(X,\omega)$. Indeed, by identification, the vertices of the squares are partitioned into 34 equivalence classes (two vertices belong to the same equivalence class if and only if they represent the same point on $(X,\omg)$). Five of them are the zeros of $\omg$, and the remaining 29 classes are regular points of $\omg$ which are marked out in the figure. The center $O$ is one of these 29 regular points. Consider the eight cylinders on $(X,\omega)$ which are the images of the two cylinders in Figure \ref{Fig:H11112} under the x-reflection and y-reflection. It is not difficult to see that the remaining 28 noncenter regular points are contained in the interior of these eight cylinders. Any other regular point which is not a vertex of the unit square is contained in one of the horizontal or vertical cylinders.
\end{example}

\section{Proof of Theorem \ref{thm:finite}}\label{sec:fin:cyl:periodic}
Our goal in this section is to prove Theorem \ref{thm:finite}.

\subsection{Torus with marked points}
\begin{lemma}\label{thm:h0k}
Let $(\mathbb{T},p_1,\cdots,p_k)\in\mathcal{H}(0^k)$ be a torus with $k\geq1$ marked points.
\begin{itemize}
  \item[(1)] If $1\leq k\leq 2$, then every point $p\in\mathbb{T}\backslash\{p_1,\cdots,p_k\}$ is contained in some simple closed geodesic of $\mathbb{T}\backslash\{p_1,\cdots,p_k\}$.
  \item[(2)] If $k\geq 3$, then  every regular point in
$\mathbb{T}\backslash\{p_1,\cdots,p_k\}$, except for a finite set, is contained in some simple closed geodesic.
\end{itemize}
\end{lemma}

\begin{proof}
Without loss of generality, we may assume that $\mathbb T$ is obtained by identifying the opposite sides of $[0,1]\times[0,1]$ such that $p$ corresponds to $(0,0)$, and $p_i=(a_i,b_i)\neq(0,0)\in[0,1]\times[0,1]$, $1\leq i\leq  k$.
\vskip 5pt
 \textbf{(1)} $1\leq k\leq 2$. It suffices to prove the statement for the case $k=2$. Consider the following three simple closed geodesics through $p$: the horizontal geodesic $\gamma_1$, the vertical geodesic $\gamma_2$, and the geodesic $\gamma_3$ with slope one. At least one of these three simple closed geodesics does not intersect $\{p_1,p_2\}$.

\medskip
\textbf{(2)} $k\geq 3$.
Let $H$ (resp. $V$) be the union of all the horizontal (resp. vertical) simple closed geodesics passing through $p_1,\dots,p_k$. Then every point in $\mathbb{T}\setminus(H\cap V)$ is contained in at least one simple closed geodesic $\gamma$ which dose not intersect $\{p_1,\dots,p_k\}$.
Since $H\cap V$ is a finite set, this completes the proof.
\end{proof}

\subsection{Periodic points}
\begin{theorem}\label{thm:no:cyl:pt:periodic}
 Let $(X,\omega)$ be a translation surface. Then any regular point which is not contained in any simple closed geodesic is a periodic point of $(X,\omega)$.
\end{theorem}

As a consequence of Theorem~\ref{thm:no:cyl:pt:periodic}, we get
\begin{corollary}\label{cor:gen:sf:non:hyp}
 Let $(X,\omega)$ be a generic surface in  a non-hyperelliptic component. Then every regular point of $(X,\omega)$ is contained in some simple closed geodesic.
\end{corollary}
\begin{proof}
By a result of Apisa~\cite[Cor.1]{Apisa_generic}, $(X,\omega)$ does not have any periodic point. Thus, the corollary follows immediately from Theorem~\ref{thm:no:cyl:pt:periodic}.
\end{proof}

Let $\mathcal L= \overline{\textup{GL}(2,\R)\cdot (X,\omg)}$, and $\mathcal L^*$  the set of  translation surfaces with one marked point $(Y,\eta,q)$, where $(Y,\eta)\in \mathcal L$ and $q\in (Y,\eta)$ is a regular point. Let $\mathcal{N}\subset \mathcal{L}^*$ be the set of  translation surfaces with one marked point $(Y,\eta,q)\in\mathcal L^*$ such that $q\in (Y,\eta)$ is not contained in any simple closed geodesic.

\begin{lemma}\label{lm:finite:cyl:closed}
The subset $\mathcal{N}$ is $\GL(2,\R)$ invariant, closed, and proper.
\end{lemma}
\begin{proof}
It is clear that $\mathcal{N}$ is $\GL(2,\R)$ invariant. The properness follows from Masur's result (\cite{Mas}) that $(X,\omega)$ contains infinitely many simple closed geodesics.
To prove that  $\mathcal{N}$ is closed, it suffices to prove that its complement  $\mathcal{N}^c$ is open. Let $(Y,\eta,q)$ be an element of $\mathcal{N}^c$. Then $q$ is contained in at least one simple closed geodesic of $(Y,\eta)$, which implies that $q$ is contained in the interior of at least one cylinder. This cylinder persists on nearby  surfaces in  $\mathcal{L}^*$. As a consequence, for all  surfaces  $(Y',\eta',q')$ sufficiently close to $(Y,\eta,q)$,  $q'$ is contained in at least one simple closed geodesic of $(Y',\eta')$. Therefore, $\mathcal{N}^c$ is open in $\mathcal L^*$.
\end{proof}

 \begin{proof}[Proof of Theorem~\ref{thm:no:cyl:pt:periodic}]
Let $\mathcal L_p\subset \mathcal{N}$ be the
  $\textup{GL}(2,\R)$ orbit closure of $ (X,\omg,p)$ in $\Lc^*$. By  Lemma~\ref{lm:finite:cyl:closed}, $\mathcal{L}_p$ is a proper affine invariant submanifold of $\mathcal{L}^*$. Therefore, $\dim \mathcal L\leq\dim \mathcal L_p<  \dim \mathcal L^*=\dim \mathcal L +1$. Hence,  $\dim \mathcal L=\dim \mathcal L_p$ and, by definition, $p$ is a periodic point of $(X,\omega)$.
  \end{proof}

\subsection{Proof of Theorem \ref{thm:finite} }

\begin{lemma}\label{lemma:bc}
Let $(X_1,\omega_1),(X_2,\omega_2)$ be two translation surfaces, and $\pi:(X_1,\omega_1)\to (X_2,\omega_2)$ a translation cover.
Let $\Sigma_2$ be the finite subset of $X_2$ that contains all the zeros of $\omega_2$ and the images of the branched points of $\pi$.
Let $\Sigma_1=\pi^{-1}(\Sigma_2)$. Assume that any point in $X_2\backslash\Sigma_2$ is contained in some simple closed geodesic which does not intersect $\Sigma_2$.
Then any point in $X_1\backslash \Sigma_1$ is also contained in some simple closed geodesic which does  not intersect $\Sigma_1$.
\end{lemma}
\begin{proof}
Observe that by assumption, $\Sigma_1$ contains all the zeros of $\omega_1$.
Let $x$  be any point of $X_1\backslash \Sigma_1$, then the image $\pi(x)$ of $x$ is a point of $X_2\backslash \Sigma_2$. By assumption, $\pi(x)$ is contained in some simple closed geodesic  $\gamma$ such that $\gamma \cap \Sigma_2=\varnothing$. On the other hand, the preimage of $\gamma$ is a finite union of simple closed geodesics of $(X_1,\omega_1)$ which do not intersect $\Sigma_1$, one of which contains $x$, say $\tilde{\gamma}$. This completes the proof.
\end{proof}

\begin{proof}[Proof of Theorem~\ref{thm:finite}]
 Let $J$ be the set of regular points on $(X,\omega)$ each of which is not contained in any simple closed geodesic. By Theorem~\ref{thm:no:cyl:pt:periodic}, every point $p\in J$ is a periodic point of $(X,\omega)$.
  By the work of Eskin-Filip-Wright \cite[Theorem 1.5]{EFW} (see also \cite[Theorem 1.2]{AW}), $(X,\omega)$ has infinitely many periodic points if and only if $(X,\omega)$ is a torus cover. Thus Theorem~\ref{thm:finite} is proved for surfaces that are not  torus covers.

  Suppose now that $(X,\omega)$ is a torus cover. Let
  $$\pi:(X,\omega)\to (\mathbb{T},\{p_1,\dots,p_k\})$$ be a translation cover, with $(\mathbb{T},\{p_1,\dots,p_k\}) \in \mathcal{H}(0^k)$ for some $k\geq1$. By Lemma \ref{thm:h0k},
   every regular point in $\mathbb{T}\setminus\{p_1,\dots,p_k\}$, except for a finite set, is contained in some simple closed geodesic. Thus in this case, Theorem~\ref{thm:finite} is a direct consequence of Lemma~\ref{lemma:bc}.
\end{proof}
\section{Proof of Theorem \ref{thm:gen:dichotomy}}
In this section, we prove Theorem \ref{thm:gen:dichotomy}.
\begin{proof}
Let $(X,\omg)\in \mathcal{H}(\kappa)$. Suppose that $p$ is a regular point of $(X, \omega)$ that is contained in some simple closed geodesic.

\vskip 5pt

\noindent \textbf{(i)}  Without loss of generality, we may assume that $(X,\omg)$ has area one, that is $(X,\omg)\in \mathcal{H}_1(\kappa)$, and $p$ is contained in a horizontal simple closed geodesic $\gamma$.  Note that $(X,\omg,p)$ belongs to the stratum $\mathcal{H}(\kappa,0)$. Let $\mathcal L_p$ be the $\SL(2,\mathbb R)$-orbit closure of $(X,\omega,p)$ in  $\mathcal{H}_1(\kappa,0)$.

 By Theorem~\ref{thm:equid},  for any interval $I=[a,b]\subset \mathbb R/2\pi\mathbb Z$ with $b\neq a$, the sector $\{a_t r_\theta \cdot (X,\omega,p) : t>0, \theta\in I\}$ is equidistributed in $\mathcal{L}_p$.
 We recall that $a_t=\left(\begin{smallmatrix}  e^t&0\\  0&e^{-t}  \end{smallmatrix}      \right)$, and
       $ r_\theta=\left( \begin{smallmatrix} \cos\theta&-\sin\theta\\  \sin\theta&\cos\theta  \end{smallmatrix}\right)$.
Consequently, there exist some sequences  $\{t_n\}_{n\in \N}$, with $t_n\to +\infty$ as $n\to\infty$, and  $\{\theta_n\}_{n\in \N} \subset I$  such that
        \begin{equation}\label{eq:limit}
       a_{t_n}r_{\theta_n} \cdot (X,\omega,p)\to (X,\omega,p) \text{ as } n\to\infty.
        \end{equation}
Let us denote
        $$a_{t_n}r_{\theta_n} \cdot (X,\omega,p)= (X_n,\omega_n, p_n).$$
Since $p$ is contained in a horizontal simple closed geodesic $\gm$ on $(X,\omg)$, it is contained in the interior of a horizontal cylinder $C$. Such a property is  preserved under small deformations, that is, there is an open neighhorhood $\mathcal{U}\subset \mathcal{H}_1(\kappa,0)$  of $(X,\omg,p)$ such that on any translation surface with one marked point $(Y,\eta,q)$ in $\mathcal{U}$, there is a  cylinder $\tilde{C}$  in the same homotopy class as $C$ whose interior contains $q$.

For $n$ sufficiently large $(X_n,\omg_n,p_n)$ belongs to $\mathcal{U}$. Hence, there is a cylinder $\tilde{C}_n$ on
$(X_n,\omg_n)$ whose interior contains $p_n$. Let $\tilde{\gamma}_n$ be the simple closed geodesic in $\tilde{C}_n$ which passes through $p_n$.

By definition, there is an affine homeomorphism $\phi_n: (X,\omg,p) \ra (X_n,\omg_n,p_n)$ whose derivative (at regular points of $(X,\omg)$) is equal to $a_{t_n}r_{\theta_n}$. In particular, we have $p_n=\phi_n(p)$.   Let $\gamma_n= \phi_n^{-1}(\tilde{\gamma}_n)$,  then $\gamma_n$ a simple closed geodesic on $(X,\omega)$ passing through $p$.

Since the forgetful map $\mathcal{H}(\kappa,0) \to \mathcal{H}(\kappa), \; (X,\omg,p) \mapsto (X,\omg)$, is continuous, we have that $(X_n,\omg_n)$ converges to $(X,\omega)$.
Since $\gamma$ and $\tilde{\gamma}_n$ belong to the same homotopy class, we get
$$
\lim\limits_{n\to\infty}\int_{\tilde{\gamma}_n} \omega_n
   =\int_{\gamma} \omega.
$$
In particular, the length of $\tilde{\gamma}_n$ on $(X_n,\omg_n)$ has an upper bound for all $n\in \N$.

By definition, we have an affine map $\varphi_n: (X,\omg,p) \ra r_{\theta_n}\cdot(X,\omg,p)$.
Let us denote by $\gamma'_n$ the image of $\gamma_n$ under $\varphi_n$.
\begin{claim}
The direction of $\gamma'_n$ on $r_{\theta_n}\cdot(X,\omega)$ converges to $\pm\pi/2\in \mathbb{RP}^1\simeq \R/\pi\Z$.
\end{claim}
\begin{proof}[Proof of the claim]
        Let $u_n+\imath v_n$ be the period of $\gamma'_n$, that is  $u_n+\imath v_n=e^{\imath\theta_n}\int_{\gamma_n}\omega$.

        We first observe that the lengths of simple closed geodesics on  $r_{\theta_n}\cdot(X,\omega)$  have a positive lower bound $\delta$ for all $n\geq1$ because $r_{\theta_n}\cdot(X,\omg)$ is isometric to $(X,\omg)$. Therefore $\sqrt{u_n^2+v_n^2} > \delta$, for all $n$.

        Suppose to a contradiction that there is a subsequence of $(u_n+\imath v_n)_{n\geq1}$, which will be denoted by $(u_n+\imath v_n)_{n\geq 1}$ for convenience, whose direction converges to some direction $\theta \neq \pm\frac{\pi}{2} \in \R/\pi\Z$. Then there exists $N'$ sufficiently large such that $\frac{|u_n|}{\sqrt{u_n^2+v_n^2}}>\frac{1}{2}|\cos\theta|$ for $n>N'$, which implies
        $$
        |u_n|>\frac{1}{2}\delta|\cos\theta|.
        $$
Since $(X_n,\omg_n)=a_{t_n}\cdot r_{\theta_n}\cdot(X,\omg)$, we have
        $$
        \left|\textup{Re}\left(\int_{\tilde{\gamma}_n}\omg_n\right)\right|=e^{t_n}|u_n|>\frac{1}{2}e^{t_n}\delta|\cos\theta|.
        $$
       As a consequence, the length of $\tilde{\gamma}_n$ on $(X_n,\omg_n)$ tends to infinity, which contradicts the property that the length of $\tilde{\gamma}_n$  converges to the length of $\gamma$ on $(X,\omega)$.
\end{proof}
       It follows from the above claim that for any $\epsilon>0$, for $n$ sufficiently large,
        the direction of $\gamma_n$ on $(X,\omega)$ is contained in some interval $]\pi/2-b-\epsilon, \pi/2-a+\epsilon[ \subset \R/\pi\Z$.
        Recall that $p$ is contained in each $\gamma_n$.
       By the arbitrariness of $I=[a,b]$ and $\epsilon$, this completes the proof of {\bf (i)}.

\medskip

\noindent \textbf{(ii)} By  \cite{KMS}, the set of directions $\theta$ such that
       the horizontal foliation of $(X, e^{\imath\theta}\omega)$ is uniquely ergodic has full measure in $\mathbb{RP}^1\simeq \mathbb R/\pi\mathbb Z$.
       Let $\theta$ be a direction in this set. In particular, there is no horizontal saddle connection on $(X,e^{\imath\theta}\omg)$.
       It follows that at least one of the half leaves of the horizontal foliation through $p$ does not meet any singularity, hence must be dense in $X$.
       Let us parametrize this half leaf by a map $L_p(t): \mathbb{R}_{\geq 0} \to X$ such that $L_p(0)=p$ and the length of the segment from $p$ to $L_p(t)$ is $t$ for any $t>0$.

       By the result of part \textbf{(i)}, the directions of the simple closed geodesics passing through $p$
       are dense in $\mathbb{RP}^1$. Thus we can take a sequence of distinct simple closed geodesics $\{\gamma_n\}$ through $p$ such that the
       direction of $\gamma_n$ converges to the horizontal direction on $(X,e^{\imath\theta}\omg)$.
       We also parametrize $\gamma_n$ by its arc length starting from $p$.
       Note that the length of $\gamma_n$ tends to infinity as $n\to\infty$.
       Since the direction of $\gamma_n$ converges to the horizontal direction, for any fixed $T$, $\gamma_n(T)\to L_p(T)$ as $n\to \infty$.

       Given any point $q$ on $(X,\omega)$ and any open neighborhood $U$ of $q$, there exists $T\geq 0$ such that $L_p(T)\in U$.
       Thus $\gamma_n$ intersects $U$ for $n$ sufficiently large.
       As a result, the union of the simple closed geodesics passing through $p$ is dense in $X$.
     \end{proof}


\section{Regular Weierstrass points on  surfaces in hyperelliptic components}\label{sec:W:pt:in:hyp:sf}
In this section, we prove a  technical result which is a generalization of \cite[Lem. 2.1]{Ngu1}
 and will be used in the proof of Theorem~\ref{thm:hyp}.
\begin{proposition}\label{prop:inv:simp:cyl}
 Let $(X,\omega)$ be a translation surface belonging to a hyperelliptic component $\Hyp(\kappa)$, where $\kappa$ is either $(2g-2)$ or $(g-1,g-1)$ with $g\geq 2$. We denote by $\tau$ the hyperelliptic involution of $X$.  Let $I$ be a saddle connection on $X$ that is invariant by $\tau$. Then there exists a simple cylinder on $X$ that contains $I$.
\end{proposition}

For our purpose, we will need the following observation which is well known to experts in the field. We include its proof here for sake of completeness.
\begin{lemma}\label{lm:emb:para}
 Let $\PP=(P_1P_2P_3P_4)$ be a parallelogram in the plane. Assume that there is a continuous map $\varphi: \PP \ra X$ satisfying the following
 \begin{itemize}
  \item[$\bullet$] $\varphi$ maps the vertices of $\PP$ to the singularities of the flat metric,

  \item[$\bullet$] for any $P\in \PP\setminus\{P_1,P_2,P_3,P_4\}$, $\varphi(P)$ is not a singularity,

  \item[$\bullet$] $\varphi$ is  locally isometric in $\inter(\PP)$.
 \end{itemize}
Then the restriction of $\varphi$ to the interior of $\PP$ is an embedding.
\end{lemma}
\begin{remark}\label{rk:emb:para}\hfill
\begin{itemize}
\item By assumption, $\varphi$ maps the sides of $\PP$ to some saddle connections of $(X,\omega)$.
 It may happen that $\varphi$ sends opposite sides of $\PP$ to the same saddle connection.

\item The assumption that no point in the interior of a side of $\PP$ is mapped to a singularity is essential. For a counter example, consider a map from a rectangle in the plane onto a horizontal cylinder whose circumference is half of the length of the horizontal sides of the rectangle.
\end{itemize}
\end{remark}

\begin{proof}[Proof of Lemma \ref{lm:emb:para}]
Without loss of generality, we can suppose that $P_1$ is the origin of plane.
Let $\overrightarrow{u}$ and $\overrightarrow{v}$ denote the vectors $\overrightarrow{P_1P_2}$ and $\overrightarrow{P_1P_4}$ respectively,
then $\PP=\{s\overrightarrow{u}+t\overrightarrow{v}, \; (s,t) \in [0;1]\times[0;1]\}$.

Assume now that there are two points $P,P' \in \inter(\PP)$ such that $\varphi(P)=\varphi(P')$. Let $\overrightarrow{w}$ denote the vector $\overrightarrow{PP'}$. We have $\overrightarrow{w}=\alpha\overrightarrow{u}+\beta\overrightarrow{v}$, with $\alpha,\beta \in (-1;1)\times(-1;1)$.
One can always find $(a,b) \in \{(0,0),(1,0),(0,1),(1,1)\}$ such that $(\alpha+a,\beta+b) \in [0;1)\times[0;1)$.
Namely, $a=0$ if $\alpha \geq 0$ and $a=1$ if $\alpha <0$, $b=0$ if $\beta \geq 0$  and $b=1$ if $\beta <0$.
Note that $\alpha+a=0$ only if $\alpha=0$, and $\beta+b=0$ only if $\beta=0$.
Since $P$ and $P'$ are two distinct points in the interior of $\PP$, we have $(\alpha,\beta)\neq (0,0)$, which implies that $(\alpha+a,\beta+b)\neq (0,0)$.

Let $Q$ be the point such that $\overrightarrow{P_1Q}=(\alpha+a)\overrightarrow{u}+(\beta+b)\overrightarrow{v}$. Since $(\alpha+a,\beta+b)\in [0;1)\times[0;1)\setminus\{(0,0)\}$, the point $Q$ is located  either in the interior of $\PP$, or in the interior of a side of $\PP$.

\begin{figure}[htb]
 \centering
 \begin{tikzpicture}[scale=0.5]
  \draw (0,0) -- (5,0) -- (5,5) -- (0,5) -- cycle;
 \draw (0,5) -- (2,2) (2,4) -- (4,1);

 \draw[thin, dashed] (0,5) -- (2,4) (2,2) -- (4,1);

 \foreach \x in {(0,0),(5,0),(5,5),(0,5)} \filldraw[fill=white] \x circle (3pt);

 \foreach \x in {(2,2),(2,4), (4,1)} \filldraw[fill=black] \x circle (2pt);

 \draw (0,0) node[below] {\tiny $P_1$} (5,0) node[below] {\tiny $P_2$} (5,5) node[above] {\tiny $P_3$} (0,5) node[above] {\tiny $P_4$};

 \draw (2,2) node[below] {\tiny $Q$} (2,4) node[right] {\tiny $P$} (4,1) node[right] {\tiny $P'$};

 \end{tikzpicture}
 \caption{Location of the point $Q$: case $(a,b)=(0,1)$}
 \label{fig:embedd:para}
\end{figure}
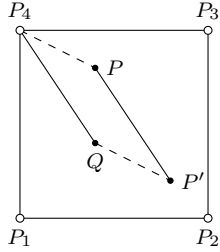
Let $P_i$ be the vertex of $\PP$ such that $\overrightarrow{P_1P_i}=a\overrightarrow{u}+b\overrightarrow{v}$. Then
$$
\overrightarrow{P_iQ}=\overrightarrow{P_1Q}-\overrightarrow{P_1P_i}=\alpha\overrightarrow{u}+\beta\overrightarrow{v}=\overrightarrow{PP'}.
$$
Now since $\varphi(P)=\varphi(P')$ and the holonomies of the flat metric are translations, the segments $\overline{P_iP}$ and $\overline{QP'}$ are mapped to the same segment in $X$, that is, the image of the parallelogram $(PP'QP_i)$ (possibly degenerated)
is contained in a cylinder. In particular, we must have $\varphi(P_i)=\varphi(Q)$. But by assumption $\varphi(P_i)$ is a singularity while $\varphi(Q)$ is a regular point of the flat metric. We thus have a contradiction which proves the lemma.
\end{proof}

\medskip

Recall that on a translation surface, a {\em separatrix} is a geodesic ray emanating  from a singularity. In the proof of Proposition~\ref{prop:inv:simp:cyl}, we will repeatedly use the following two lemmas

\begin{lemma}\label{lm:inv:embed:para:2}
Let $(X,\omega)$ be a translation surface in a hyperelliptic component, and $I$ be a horizontal saddle connection on $X$ which is invariant by the hyperelliptic involution. Assume that some vertical separatrices of $X$ intersect $\inter(I)$. Then there is a parallelogram $\PP$ in the plane, and a map $\varphi: \PP \rightarrow X$ such that
\begin{itemize}
 \item[(a)] no side of $\PP$ is vertical, nor horizontal,

 \item[(b)] $\PP$ contains a horizontal diagonal which is mapped onto $I$.

 \item[(c)] $\varphi$ is locally isometric and maps each side of $\PP$ to a saddle connection,

 \item[(d)] the restriction of $\varphi$ into $\inter(\PP)$ is an embedding.
\end{itemize}
\end{lemma}

\begin{proof}
Consider the set of vertical separatrices that intersect $\inter(I)$.  For each of these rays, we consider the segment from its origin to its first intersection with $\inter(I)$.
We then get a finite family of vertical geodesic segments, each of which joins a singularity to a point in $\inter(I)$, and contains no point of $\inter(I)$ in the interior.
Let $\eta$ be a segment of minimal length in this family.

We can identify $I$ with a horizontal segment $\ol{PQ}$ in the plane, and $\eta$ with a vertical segment in $\R^2$ whose endpoints are denoted be $P_1$ and $R_1$, where $R_1 \in \inter(\ol{PQ})$.
Let $Q_1$ denote the image of $P_1$  under the
rotation by $\pi$ around the midpoint of  $\ol{PQ}$.
Let $\PP$ denote the parallelogram with vertices $P,P_1,Q,Q_1$.

Using the developing map (for instance by following the vertical lines through the points in $\ol{PQ}$), we can define a locally isometric map from a rectangle $\ol{PQ}\times (-\epsilon,\epsilon)$ to $(X,\omega)$, for some $\epsilon>0$. We claim that this local isometry can be extended to the rectangle $\RR$ (bounded by vertical and horizontal segments) that contains $P,Q,P_1,Q_1$ in the boundary. Indeed, the extension to $\RR$ fails only if there is a point in $\inter(\RR)$ that gets mapped to a singularity of $X$. But this implies that there is a vertical geodesic segment from that singularity to a point in $\inter(I)$ which is shorter than $\eta$. By the choice of $\eta$, this is impossible.

Since $\PP \subset \RR$, we get a locally isometric map $\varphi: \PP \ra X$.
By construction, no side of $\PP$ is vertical, nor horizontal, and we have $\varphi(\ol{PQ})=I$.

Note that  $\varphi(\ol{PP_1})$ and $\varphi(\ol{QP_1})$ are saddle connections of $(X,\omega)$ because $P_1$ is mapped to a singularity.
Since the hyperelliptic involution  fixes the midpoint of $I$ and permutes its endpoints, this involution can be identified with the 
rotation by $\pi$ around the midpoint of  $\ol{PQ}$ via $\varphi$.
Consequently, $\varphi(\ol{Q_1Q})$ and $\varphi(\ol{Q_1P})$ are also saddle connections of $(X,\omega)$.
It follows that the map $\varphi$ satisfies the hypothesis of Lemma~\ref{lm:emb:para}, hence the restriction of $\varphi$ to $\inter(\PP)$ is an embedding.
\end{proof}

\medskip

\begin{lemma}\label{lm:pair:sc:cut}
Let $(X,\omega)$ be a translation surface in a hyperelliptic component $\Hyp(\kappa)$. Let $(\delta^+,\delta^-)$ be a pair of saddle connections on $X$ that are permuted by the hyperelliptic involution $\tau$. Then removing $\delta^+$ and $\delta^-$ decomposes $X$ into two (connected) components  invariant by $\tau$, the closures of these two components, denoted by  $X'$ and $X''$, are subsurfaces of $X$  bounded by $\delta^+\cup\delta^-$. The closed surfaces $\hat{X}'$ and $\hat{X}''$ obtained from $X'$ and $X''$ respectively by gluing $\delta^+$ and $\delta^-$ together  belong to some hyperelliptic components of lower dimension than $\Hyp(\kappa)$.
\end{lemma}
\begin{proof}
By definition, there is a double cover $\pi: X \ra \Sph^2$ branched over $2g+2$ points (where $g$ is the genus of $X$), and a quadratic differential $q$ on $\Sph^2$ such that $\omega^2=\pi^*q$. If $g\geq 2$ then $q$ has a unique zero and $2g+1$ simple poles. For $g=1$, $q$ has four simple poles. If $\kappa=(0)$, then  one of the poles is the image of the unique marked point of $X$, and if $\kappa=(0,0)$, then the image of the two marked points on $X$ is a regular marked point. By a slight abuse, in the  case $g=1$, we will consider the image of the marked point(s) on $X$  as the unique zero of $q$. In all cases, let us denote by $y_0$ the unique zero of $q$.

Since $\delta^+$ and $\delta^-$ are permuted by $\tau$, both of them are mapped by $\pi$ onto a simple closed curve $\delta$ through $y_0$ on $\Sph^2$. Note that $\delta$ is a saddle connection for the flat metric defined by $q$. The curve $\delta$ decomposes $\Sph^2$ into two subsurfaces, denote by $Y', Y''$, whose boundary is identified with $\delta$. Note that $Y'$ and $Y''$ are both homeomorphic to a closed disc.

Let $X'=\pi^{-1}(Y'), X''=\pi^{-1}(Y'')$. Then $X'$ and $X''$ are  connected subsurfaces of $X$ invariant under $\tau$ whose boundary is $\delta^+\cup\delta^-$. This proves the first assertion of the lemma.

Let $y$ be the midpoint of $\delta$, that is, $y$ decomposes $\delta$ into two subsegments $\delta_1, \delta_2$ of the same length.
Let $\hat{Y}'$ (resp.  $\hat{Y}''$) be the closed flat surface obtained from $Y'$ (resp. $Y''$) by gluing $\delta_1$ and $\delta_2$ together so that $y$ becomes a singularity with cone angle equal to $\pi$. Since $\hat{Y}'$ and $\hat{Y}''$ are homeomorphic to the sphere, and all the angles at the singularities of $\hat{Y}'$ and $\hat{Y}''$ are integral multiples of $\pi$, the flat metric of $\hat{Y}'$ and $\hat{Y}''$ are  defined by some quadratic differentials $q'$ and $q''$ respectively. By construction, $q'$ and $q''$ have at most one zero, which must be located at $y_0$.

It is not difficult to see that the surface $\hat{X}'$ (resp. $\hat{X}''$) is actually the holonomy double cover of $\hat{Y}'$ (resp. $\hat{Y}''$). Since $q'$ and $q''$ have at most one zero, $\hat{X}'$ and $\hat{X}''$ belong to some hyperelliptic components, and the second assertion of the lemma follows.
\end{proof}

\begin{remark}\label{rk:hyperelliptic:comp}
If $X$ is a hyperelliptic Riemann surface, but $(X,\omega)$ does not belong to a hyperelliptic component, then the conclusion of Lemma~\ref{lm:pair:sc:cut} does not hold. This is because in this case $q$ has more than one zero, hence the image of $\delta^+$ and $\delta^-$ may be a simple arc on the sphere.
\end{remark}

\medskip

\begin{proof}[Proof of Proposition \ref{prop:inv:simp:cyl}]
Without loss of generality, we can assume that $I$ is horizontal. We identify $I$ with a segment $\ol{PQ}$ in $\R^2$.
It is a classical result that for all but countably many directions $\theta$, the flow in direction $\theta$ is minimal on $(X,\omega)$ (see for instance \cite{MT}).
Pick a non-horizontal minimal direction $\theta$ for $(X,\omega)$. Using the action of $\{\left(\begin{smallmatrix}  1 & t \\ 0 & 1 \end{smallmatrix}\right), \; t \in \R\}$, we can assume that $\theta$ is the vertical direction.
Since the vertical flow is minimal, every vertical separatrix  is dense in $X$. In particular,  they all intersect $\inter(I)$.

Let $\PP_1=(PP_1QQ_1)$ be the parallelogram, and $\varphi_1:\PP_1 \ra X$ be the locally isometric map constructed in Lemma~\ref{lm:inv:embed:para:2}.
Let $\delta_1^+:=\varphi_1(\ol{PP_1})$, $\delta^-_1:=\varphi_1(\ol{Q_1Q})$, $\gamma^+_1:=\varphi_1(\ol{QP_1})$, and $\gamma^-_1:=\varphi_1(\ol{Q_1P})$.
Note that $\delta_1^\pm$ and $\gamma_1^\pm$ are saddle connections of $(X,\omega)$, and that $\delta_1^+$ and $\delta^-_1$ (resp. $\gamma_1^+$ and $\gamma_1^-$) are exchanged by the hyperelliptic involution $\tau$.

If $\delta_1^+=\delta^-_1$ and $\gamma^+_1=\gamma^-_1$ then  $\varphi_1(\PP_1)$ is a closed torus, which is impossible since we have assumed that $X$ has genus at least two. Therefore, up to a relabeling of $\{P_1,Q_1\}$, we can suppose that $\delta_1 ^+\neq \delta_1^-$.

\medskip

We now consider the following algorithm: assume that we have a locally isometric map $\varphi_n$ from a parallelogram $\PP_n=(PP_nQQ_n) \subset \R^2$ to $X$ that contains $\ol{PQ}$ as a diagonal such that
\begin{itemize}
 \item[(i)] $\varphi_n(\ol{PQ})=I$,

 \item[(ii)] $\varphi_n$ maps the sides of $\PP_n$ to saddle connections of $(X,\omega)$,

  \item[(iii)] the restriction of $\varphi_n$ to the interior of $\PP_n$ is an embedding,

  \item[(iv)] $\delta^+_n:=\varphi_n(\ol{PP_n})$ and $\delta^-_n:=\varphi_n(\ol{QQ_n})$ are two different saddle connections.
\end{itemize}
Let us denote by $\gamma^+_n$ and $\gamma^-_n$ the saddle connections of $(X,\omega)$ which are the images  of $\ol{QP_n}$ and $\ol{PQ_n}$ under $\varphi_n$ respectively.

If $\gamma^+_n=\gamma^-_n$ then the algorithm stops. In this case, $\varphi_n(\PP_n)$ is a simple cylinder (bounded by $\delta^+_n$ and $\delta^-_n$) that contains $I$ and we are done.
Otherwise, we use the action of a matrix $U_t=\left(\begin{smallmatrix}  1 & t \\ 0 & 1 \end{smallmatrix}\right),  t \in \R$, to make $\ol{PP_n}$ a vertical segment.
By a slight abuse of notation, we continue denoting by $\PP_n$ and $(X,\omega)$ the images of $\PP_n$ and of  $(X,\omega)$ under the action of $U_t$ respectively.
Consequently, there is a map from $\PP_n$ to $(X,\omega)$ with the same properties as $\varphi_n$. We denote this map also by $\varphi_n$.
It is worth noticing that, even though in each step we actually change the surface $(X,\omega)$ by applying to it some matrix in $\{\left(\begin{smallmatrix}  1 & t \\ 0 & 1 \end{smallmatrix}\right), \; t \in \R\}$, the conclusion we get also holds for the original surface because we are staying in the same $\GL(2,\R)$-orbit.

\medskip

We then have two cases:
\begin{itemize}
 \item[$\bullet$] {\bf Case 1:} some vertical separatrices intersect $\inter(I)$. In this case,  using Lemma~\ref{lm:inv:embed:para:2}, we obtain a locally isometric map $\varphi_{n+1}$ from parallelogram $\PP_{n+1}=(PP_{n+1}QQ_{n+1})$, that contains $\ol{PQ}$ as a diagonal to $X$. We choose the labeling such that $P_{n+1}$ and $P_n$ are on the same side with respect to $I$ (see Figure~\ref{fig:algo:inv:simple:cyl}).

By construction, $\varphi_{n+1}$ satisfies the properties (i),(ii), (iii). We claim that $\varphi_{n+1}$  satisfies (iv) as well.
Let $\delta^+_{n+1}:=\varphi_{n+1}(\ol{PP_{n+1}})$ and $\delta^-_{n+1}:=\varphi_{n+1}(\ol{QQ_{n+1}})$.
Suppose that $\delta^+_{n+1}=\delta^-_{n+1}$. Because $\ol{PP_{n+1}}$ and $\ol{Q_{n+1}Q}$ are mapped to the same saddle connections, the geodesic rays in direction $(0,1)$ starting from $P$ and from $Q_{n+1}$ are mapped to the same ray in $X$. Let $Q'$ be the point in the plane such that $\overrightarrow{Q_{n+1}Q'}=\overrightarrow{PP_n}$.

If $Q'\in \inter(\PP_{n+1})$ then $\varphi_{n+1}(\ol{Q_{n+1}Q'})=\delta_n^+$, but this is a contradiction because points in $\inter(\PP_{n+1})$ are mapped to regular points in $X$. Thus $Q'\not\in\inter(\PP_{n+1})$. Consequently, the segment $\ol{Q_{n+1}Q'}$ must intersect $\inter(\ol{PQ})$, which implies that $\delta_n^+$ intersects $\inter(I)$ contradicting the hypothesis that $\varphi_n$ is an embedding in $\inter(\PP_n)$. Therefore, we must have $\delta_{n+1}^+\neq \delta_{n+1}^-$.

 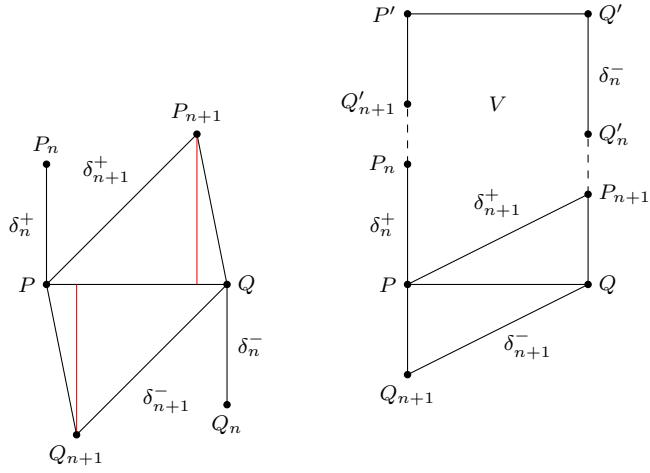
\begin{figure}[htb]
 \begin{tikzpicture}[scale=0.4]
  \draw (-9,4) -- (-9,0) -- (-3,0) -- (-3,-4);
  \draw (-9,0) -- (-4,5) -- (-3,0) -- (-8,-5) -- cycle;
  \draw[red] (-8,0) -- (-8,-5)  (-4,5) -- (-4,0);
  \foreach \x in {(-9,4),  (-9,0), (-8,-5), (-4,5), (-3,0), (-3,-4)} \filldraw \x circle(3pt);
  \draw (-9,4) node[above] {\tiny $P_n$} (-9,0) node[left] {\tiny $P$} (-8,-5) node[below] {\tiny $Q_{n+1}$} (-4,5) node[above] {\tiny $P_{n+1}$} (-3,0) node[right] {\tiny $Q$} (-3,-4) node[below] {\tiny $Q_n$};
  \draw (-9,2) node[left] {\tiny $\delta^+_n$} (-3,-2) node[right] {\tiny $\delta^-_n$};
  \draw (-7,3.8) node {\tiny $\delta^+_{n+1}$} (-5,-3.8) node {\tiny $\delta^-_{n+1}$};

  \draw (3,6) -- (3,9)  -- (9,9) -- (9,5);
  \draw (3,4) -- (3,-3) -- (9,0) -- (9,3) -- (3,0) -- (9,0);
  \draw[dashed] (3,4) -- (3,6) (9,3) -- (9,5);
  \foreach \x in {(3,9),(3,6), (3,4), (3,0), (3,-3), (9,9), (9,5), (9,3), (9,0)} \filldraw[black] \x circle (3pt);

   \draw (3,9) node[left] {\tiny $P'$}  (3,6) node[left] {\tiny $Q'_{n+1}$}   (3,4) node[left] {\tiny $P_n$} (3,0) node[left] {\tiny $P$} (3,-3) node[below] {\tiny $Q_{n+1}$};
   \draw (9,9) node[right] {\tiny $Q'$} (9,5) node[right] {\tiny $Q'_n$} (9,3) node[right] {\tiny $P_{n+1}$} (9,0) node[right] {\tiny $Q$};
   \draw (3,2) node[left] {\tiny $\delta^+_n$} (9,7) node[right] {\tiny $\delta^-_n$};
   \draw (6,2.7) node {\tiny $\delta^+_{n+1}$} (7,-2) node {\tiny $\delta^-_{n+1}$};
   \draw (6,6) node {\tiny $V$};

 \end{tikzpicture}

\caption{Embedded parallelogram containing $I$: Case 1 (left) and Case 2 (right).}
\label{fig:algo:inv:simple:cyl}
\end{figure}
\item[$\bullet$] {\bf Case 2:} no vertical separatrix intersects $\inter(I)$. In this case $I$ is contained in a vertical cylinder $V$ whose boundary contains $\delta^+_n$ and $\delta^-_n$.
Let $\ell$ be the circumference of $V$. We can represent $V$ by a rectangle $\RR=(PQQ'P')$ in the plane as shown in Figure~\ref{fig:algo:inv:simple:cyl}, where the length of $\ol{PP'}$ and $\ol{QQ'}$ is $\ell$.
There is a locally isometric mapping $\varphi: \RR \ra X$ such that $\varphi(\ol{PQ})=\varphi(\ol{P'Q'})=I$, and the restriction of $\varphi$ to $\inter(\RR)$ is an embedding.

The vertical sides of $\RR$ are decomposed into several segments, each of which is mapped to a saddle connection of $(X,\omega)$.
By construction $\ol{PP_n}$ is the bottom most segment in the left border of $\RR$.
Let $Q'_n$ be the lower endpoint of the topmost segment in the right border of $\RR$, then $\ol{Q'Q'_n}$ is mapped to the saddle connection $\delta^-_n$.

If $V$ is a simple cylinder then $\gamma_n^+=\gamma_n^-$ and we are done. Thus, let us suppose that the left  border of $V$ contains  a saddle connection other than $\delta^+_n$.
Since $V$ is invariant under the hyperelliptic involution (because $I$ is), its right border must contain the same number of saddle connections as its left border.
Therefore, the right border must contain a saddle connection other than $\delta^-_n$.

Let $Q'_{n+1}$ be the lower endpoint of the topmost segment in the left border of $\RR$, and $P_{n+1}$ be the upper endpoint of the bottommost segment in the right border of $\RR$.  Then the saddle connection which is the image of $\ol{P'Q'_{n+1}}$ is not $\delta^+_n$, and the one which is the image of $\ol{QP_{n+1}}$ is not $\delta^-_n$.

Let $Q_{n+1}$ be the image of $Q'_{n+1}$ under the translation by $\overrightarrow{P'P}$. Consider the parallelogram $\PP_{n+1}=(PP_{n+1}QQ_{n+1})$.
The map from $\RR$ to $(X,\omega)$ can be extended to the triangle $(PQQ_{n+1})$, hence we have a locally isometric map $\varphi_{n+1}: \PP_{n+1} \ra (X,\omega)$.
Let us denote by $\delta^+_{n+1}$ and $\delta^-_{n+1}$ the images under $\varphi_{n+1}$ of $\ol{PP_{n+1}}$ and $\ol{QQ_{n+1}}$ respectively.
It is straightforward to check that $\varphi_{n+1}$ satisfies the properties (i),\dots,(iv).
\end{itemize}

\medskip

We now show  that the algorithm has to  stop after finitely many steps.
For $i\in\{n,n+1\}$, since $\delta^+_{i}$ and $\delta^-_{i}$ are permuted by $\tau$, it follows from Lemma~\ref{lm:pair:sc:cut}
that $(\delta^+_{i},\delta^-_{i})$ decomposes $X$ into two subsurfaces invariant by $\tau$.
Let $X_i$ denote the subsurface that contains $I$.

%

Let $\vartheta(X_i)$ be the total cone angle at the singularities of $X$ inside $X_i$.
To be more precise, if $\omega$ has one zero, then $\delta_i^+,\delta_i^-$ divide a small disk about the unique zero of $\omega$ into 4 sectors. In this case $\vartheta(X_i)$ is the sum of the angles of the two sectors contained in $X_i$. If $\omega$ has two zeros, then $\delta^+_i,\delta^-_i$ give two rays at each zero. In this case $\vartheta(X_i)$ is the sum of the angles of the sectors (one at each zero) that are contained in $X_i$.

Note that $\vartheta(X_i)$ is also the total angle at the singularities of the translation surface $\hat{X}_i$ which is obtained from $X_i$ by gluing $\delta^+_i$ and $\delta^-_i$ together.
Therefore, $\vartheta(X_i)$ must be an integral multiple of $2\pi$. Note also that $\vartheta(X_i)=2\pi$ if and only if $X_i$ is a simple cylinder, in which case $\hat{X}_i$ is a torus with a marked point.

\medskip

If $\delta^+_n$ and $\delta^+_{n+1}$ intersect at a point in their interior, then there is a point in $\inter(\ol{PP_{n+1}})$ which is mapped to a point in $\delta^+_n$. By considering the vertical ray starting from this point in direction $(0,-1)$, we see that $\delta^+_n$ must intersect $\inter(I)$. This is because this vertical ray must terminate at a singularity, but all the points in $\inter(\PP_{n+1})$ are mapped to regular points of $X$, hence it must cross $\PP_{n+1}$ entirely. In particular it must intersect $\inter(\ol{PQ})$.
Since this contradicts the assumption that $\varphi_n$ is an embedding in $\inter(\PP_n)$, we conclude that $\delta^+_{n+1}$ and $\delta^+_n$ can only meet at their endpoints. Similarly, $\delta^-_{n+1}$ and $\delta^+_n$ can only meet at their endpoints as well. Since $\delta^-_n$ is the image of $\delta^+_n$ under the hyperelliptic involution, the same statements hold for $\delta^-_n$.

It follows that $X_{n+1}$ is a proper subsurface of $X_n$,  and therefore $\vartheta(X_{n+1})< \vartheta(X_n)$.
Consequently, the algorithm has to stop after at most $\vartheta(X)/(2\pi)$ steps, where $\vartheta(X)$ is the total cone angle at the singularity of $X$, and we get a simple cylinder that contains $I$.
\end{proof}

\begin{corollary}\label{cor:w:pt:sim:cyl}
Let $(X,\omega)$ be a translation surface in some hyperelliptic component $\Hyp(\kappa)$ of genus at least $2$. Let $W$ be a Weierstrass point of $X$ which is not a zero of $\omega$.
Then $W$ is contained in a simple cylinder of $(X,\omega)$.
\end{corollary}
\begin{proof}
Let $J$ be a geodesic segment of that realizes the distance from $W$ to the zero set of $\omega$. Then $I:=J\cup\tau(J)$ is a saddle connection invariant by $\tau$.
Thus, by Proposition~\ref{prop:inv:simp:cyl}, there is a simple cylinder $C$ in $X$ that contains $I$.
\end{proof}

\section{Proof of Theorem~\ref{thm:hyp}}
In this section, we will prove Theorem~\ref{thm:hyp}.
Let $\mathcal{L}$ be the $\GL(2,\R)$ orbit closure of $(X,\omega)$ in $\mathcal{H}^{\rm hyp}(\kappa)$ with $\kappa\in\{(2g-2),(g-1,g-1)\}, \; g \geq 2$. By Theorem \ref{thm:Apisa}, there are five possibilities:
\begin{enumerate}[(I)]
\item   $\mathcal{L}=\mathcal{H}^{\rm hyp}(\kappa)$, where $\kappa=(2g-2)$ or $(g-1,g-1)$.\\

\item  $4\leq\dim \mathcal{L} < \dim \mathcal{H}^{\rm hyp}(\kappa)$ and $\mathcal{L}$ is obtained from a translation covering construction over $\mathcal{H}^{\rm hyp}(\kappa')$, that is, surfaces in $\mathcal{L}$ are translation covers of surfaces in $\mathcal{H}^{\rm hyp}(\kappa')$, where $\kappa'=(2r-2)$ or $(r-1,r-1)$ and $2\leq r<g$. In this case, we will write $\mathcal{L}= \widetilde{\mathcal{H}^{\rm hyp}}(\kappa')$.\\

\item  $\dim \mathcal{L}=3$ and $\mathcal{L}$ consists of translation covers of surfaces in some eigenform locus $\Omega E_D(1,1)$ in $\mathcal{H}(1,1)$, where the associated discriminant $D$ is not a square. In this case we will write $\mathcal{L}=\widetilde{\Omega E}_D(1,1)$.\\

\item  $\dim \mathcal{L}=3$ and $\mathcal{L}$ consists of translation covers of surfaces in $\mathcal{H}(0,0)$, the covering maps are ramified over two points. In this case we will write $\mathcal{L}=\widetilde{\mathcal{H}}(0,0)$.\\

\item  $\dim \mathcal{L}=2$, $\mathcal{L}$ is a closed orbit.
\end{enumerate}

\begin{remark}\label{rk:numbering:prf:Th3}\hfill
\begin{itemize}
 \item Case (4) and (5) in Theorem~2.9 are regrouped in Case (II) above.

 \item The numbering is chosen to better fit the strategy of the proof.
\end{itemize}
\end{remark}

\subsection{Case I: $\mathcal{L}=\mathcal{H}^{\rm hyp}(\kappa)$}\label{ssec:I}

\begin{proof}[Proof of Theorem~\ref{thm:hyp} for Case I]
Suppose to a contradiction that there is a regular point $p$ of $(X,\omega)$ which is not contained in any closed geodesic. By Theorem~\ref{thm:no:cyl:pt:periodic}, $p$ is a periodic point of $(X,\omega)$. Therefore, $p$ is a Weierstrass point of $(X,\omega)$ by Theorem~\ref{thm:pt}.
But by Corollary~\ref{cor:w:pt:sim:cyl}, $p$ is contained in a simple cylinder.
Thus we have a contradiction, which proves the theorem for this case.
\end{proof}

\subsection{Case II: $\mathcal{L}=\widetilde{\mathcal{H}^{\rm hyp}}(\kappa')$ for some  $\kappa'=(2r-2)$ or $(r-1,r-1)$ with $2\leq r<g$}\label{ssec:II}

\begin{proof}[Proof of Theorem~\ref{thm:hyp} for Case II]
Let  $\pi:(X,\omega)\to(X',\omega')$ be a translation cover with $(X',\omega') \in {\mathcal{H}^{\rm hyp}(\kappa')}$.
Let $\Sigma$ and $\Sigma'$ be the zero sets of $\omega$ and $\omega'$ respectively.
In this case we have $\Sigma=\pi^{-1}(\Sigma')$.
By assumption, the $\GL(2,\R)$ orbit closure of $(X',\omega')$ is ${\mathcal{H}^{\rm hyp}(\kappa')}$.
By Theorem~\ref{thm:hyp} for Case I, any regular point of $(X',\omega')$ is contained in infinitely many simple closed geodesics.
Lemma~\ref{lemma:bc} then implies that the same is true for $(X,\omega)$.
\end{proof}

\subsection{Case III: $\mathcal{L}=\widetilde{\Omega E}_D(1,1)$, $D$ is not a square.}\label{ssec:III}
\begin{proof}[Proof of Theorem~\ref{thm:hyp} for Case III]
Let $\pi: (X,\omega)\to (Y,\eta)\in {\Omega E_D(1,1)}$ be a translation cover which is branched over the zeros of $\eta$.
Then $\overline{\GL(2,\R)\cdot (Y,\eta)}={\Omega E_D(1,1)}$.
By Lemma \ref{lemma:bc}, to prove Theorem \ref{thm:hyp}, it suffices to prove that every regular point of $(Y,\eta)$ is contained in some simple closed geodesic of $(Y,\eta)$.
Suppose to a contradiction that $(Y,\eta)$ contains a regular point $q$ which is not contained in any simple closed geodesic.
Then, by Theorem~\ref{thm:no:cyl:pt:periodic}, $q$ is a periodic point of $(Y,\eta)$. If $q$ is a Weierstrass point, by Corollary~\ref{cor:w:pt:sim:cyl}, $q$ is contained in a simple cylinder and we have a contradiction.

If $q$ is not a Weierstrass point of $(Y,\eta)$, then by Theorem~\ref{thm:pt}, $\overline{\GL(2,\R)\cdot (Y,\eta)}$ is the golden eigenform locus.
It follows from \cite[Sect. 2]{Apisa_periodic_g2}, that the $\GL(2,\R)$ orbit closure of $(Y,\eta,q)$ contains a marked translation surface $(Y',\eta',q')$, where $(Y',\eta')$ is a Veech surface, and $q'$ is a regular Weierstrass point.
By Lemma~\ref{lm:finite:cyl:closed}, $q'$ is not contained in any simple closed geodesic of $(Y',\eta')$, which again contradicts Corollary~\ref{cor:w:pt:sim:cyl}.
\end{proof}

\begin{remark}\label{rmk:other:prf:caseIII}
For this case, the theorem can also be shown by more direct arguments, using the fact that surfaces in eigenform loci are completely periodic in the sense of Calta.
\end{remark}

\subsection{Case V: $\mathcal{L}$ is a closed orbit, that is $(X,\omega)$ is a Veech surface}\label{ssec:V}
Let us start with the following observation
\begin{lemma}\label{lm:sim:cyl:in:other:cyl}
 Let $(X,\omega)$ be a translation surface in  $\mathcal{H}^{\rm hyp}(2g-2)$ or $\mathcal{H}^{\rm hyp}(g-1,g-1)$. Let $\gamma$ be  a saddle connection of $(X,\omega)$ which is invariant under the hyperelliptic involution.
 Assume that $\gamma$ is contained in the boundary of a cylinder $C$. Then $\gamma$ is contained in a simple cylinder $D$ such that $D\subset \ol{C}$, and the core curves of $D$ cross $C$ once.
 \end{lemma}
 \begin{proof}
Applying a rotation if necessary, we can assume that $C$ is horizontal. Without loss of generality , we can assume that $\gamma$ is contained in the top boundary of $C$. Since $(X,\omega)$ belongs to a hyperelliptic component, $C$ is invariant under the hyperelliptic involution (see for instance \cite{KZ,NW}). This implies that $\gamma$ is also contained in the bottom boundary of $C$.
Therefore, there are (infinitely many) simple closed geodesics  in $\ol{C}$ that join the midpoint of $\gamma$ to itself and cross every core curve of $C$ once. The cylinder corresponding to any  such closed geodesic has the required properties.
\end{proof}

We now show
\begin{proposition}\label{prop:veechsf}
Let $(X,\omega)$ be a Veech surface in $\mathcal{H}^{\rm hyp}(\kappa)$ where $\kappa=(2g-2) $ or $(g-1,g-1)$. Then every regular point in  $(X,\omega)$ is contained in some simple closed geodesic.
\end{proposition}
\begin{proof}\hfill

\noindent \underline{\textbf{Step 1.}}
Let $w$ be a regular Weierstrass point, let $J$ be a geodesic segment of that realizes the distance from $w$ to the zero set of $\omega$. Then $\gamma_0:=J\cup\tau(J)$ is a saddle connection invariant by $\tau$.
Since $(X,\omega)$ is a Veech surface, it is periodic in the direction of $\gamma_0$.
Thus $\gamma_0$ is contained in the boundary of some cylinder $C_0$.
By Lemma~\ref{lm:sim:cyl:in:other:cyl}, there is a simple cylinder $D_0$ in $\ol{C}_0$ that contains $\gamma_0$.
Let $\gamma^+_1,\gamma^-_1$ be the two saddle connections bordering $D_0$.
Then the saddle connections $\gamma_0,\gamma^+_1,\gamma^-_1$ bound an embedded parallelogram $P_0$.
By construction every regular point in $\ol{P}_0$ is contained in a closed geodesic (regular points in $\gamma^\pm_1$ and in $\inter(P_0)$ are contained in $C_0$, while regular points in $\gamma_0$ are contained in $D_0$).

\medskip

\noindent \underline{\textbf{Step 2.}}
Note that the hyperelliptic involution preserves $D_0$, exchanges $\gamma^+_1$ and $\gamma^-_1$, hence it preserves the complement of $D_0$.
Remove $\inter(D_0)$ from $X$ and identify $\gamma^+_1$ with $\gamma^-_1$, we get a closed translation surface $(X_1,\omega_1)$ in some hyperelliptic component of lower dimension (c.f. Lemma~\ref{lm:pair:sc:cut}). Let $\gamma_1$ denote the saddle connection corresponding to $\gamma^+_1$ and $\gamma^-_1$ on $X_1$. Then $\gamma_1$ is invariant under the hyperelliptic involution of $(X_1,\omega_1)$.

Since $(X,\omega)$ is a Veech surface, it is periodic in the direction of $\gamma^\pm_1$. Therefore, $(X_1,\omega_1)$ is periodic in this direction.
 It follows that $\gamma_1$ is contained in the boundary of a cylinder $C_1$ of $(X_1,\omega_1)$.
 Again by Lemma~\ref{lm:sim:cyl:in:other:cyl}, $\gamma_1$ is contained in a simple cylinder $D_1 \subset \ol{C}_1$ bounded by two saddle connections $\gamma^+_2,\gamma^-_2$.
 By construction every point in $\inter(\gamma^+_2)\cup\inter(\gamma^-_2)$ is contained in $C_1$.

 In $X$, $D_1$ corresponds to an embedded parallelogram $P_1$ which is bounded by $\gamma^\pm_1$ and $\gamma^\pm_2$.
 By construction, every regular point in $\ol{P}_1$ is contained in a simple closed geodesic of $(X,\omega)$.

 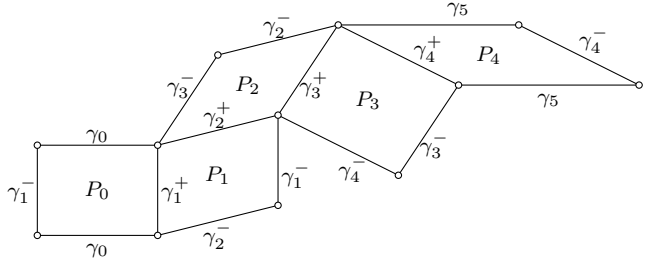
\begin{figure}[htb]
 \centering

 \begin{tikzpicture}[scale=0.4]
  \draw (0,0) -- (4,0) -- (8,1) -- (8,4) -- ( 12,2) -- (14,5) -- (20,5) -- (16,7) -- (10,7) -- (6,6) -- (4,3) -- (0,3) -- cycle;
  \draw (4,0) -- (4,3) -- (8,4) -- (10,7) -- (14,5);

  \foreach \x in {(0,0), (0,3), (4,0), (4,3), (6,6), (8,1), (8,4), (10,7), (12,2), (14,5), (16,7), (20,5)} \filldraw[fill=white] \x circle (3pt);

  \draw (2,3.3) node {\tiny $\gamma_0$} (2,-0.5) node {\tiny $\gamma_0$};
  \draw (-0.5, 1.5) node {\tiny $\gamma^-_1$} (4.6,1.5) node {\tiny $\gamma^+_1$} (8.6,2) node {\tiny $\gamma^-_1$};
  \draw (6,0) node {\tiny $\gamma^-_2$} (6,4) node {\tiny $\gamma^+_2$} (8,7) node {\tiny $\gamma^-_2$};
  \draw (4.8,5) node {\tiny $\gamma^-_3$} (9.2,5) node {\tiny $\gamma^+_3$} (13.2,3) node {\tiny $\gamma^-_3$};
  \draw (10.5,2.2) node {\tiny $\gamma^-_4$} (13,6.2) node {\tiny $\gamma^+_4$} (18.5,6.5) node {\tiny $\gamma^-_4$};
  \draw (14,7.5) node {\tiny $\gamma_5$} (17,4.5) node {\tiny $\gamma_5$};
  \draw (2,1.5) node  {\tiny $P_0$};
  \draw (6,2) node  {\tiny $P_1$};
  \draw (7,5) node  {\tiny $P_2$};
  \draw (11,4.5) node  {\tiny $P_3$};
  \draw (15,6) node  {\tiny $P_4$};
 \end{tikzpicture}

  \caption{Decomposition of $(X,\omega)$ into parallelograms.}
  \label{fig:para:dec:veech}
 \end{figure}

\medskip

 \noindent\underline{\bf Step 3.}
 By repeating Step 2, we get a sequence of pairs of saddle connections $(\gamma_i^+,\gamma_i^-), \; i=1,2,\dots$ as follows:
 in the $i$-th step, we have a pair of saddle connections $(\gamma^+_i,\gamma^-_i)$ permuted by the hyperelliptic involution that decompose $X$ into two components (c.f. Lemma~\ref{lm:pair:sc:cut}), one of which contains $\gamma_0$ and all the pairs $(\gamma^+_1,\gamma^-_1),\dots,(\gamma^+_{i-1},\gamma^-_{i-1})$.
 Consider the other component. Its closure  is a subsurface of $X$ bounded by $\gamma^+_i\cup\gamma^-_i$. Denote by $(X_{i},\omega_i)$ the translation surface obtained from this subsurface by gluing $\gamma^+_i$ and $\gamma^-_i$ together. By Lemma~\ref{lm:pair:sc:cut}, $(X_i,\omega_i)$ belongs to some hyperelliptic component. Note that the identification of $\gamma^+_i$ and $\gamma^-_i$ gives rise to a saddle connection $\gamma_i$ in $X_{i}$ that is invariant by the hyperelliptic involution.

 Since $(X,\omega)$ is a Veech surface, the direction of $\gamma^\pm_i$ is periodic. Recall that $X_{i}\setminus\gamma_i$ is a subsurface of $X$, therefore $X_{i}$ is decomposed into cylinders in the direction of $\gamma_i$. Since $(X_i,\omega_i)$ is in a hyperelliptic component, there is a unique cylinder $C_i$ in this family that contains $\gamma_i$ it its boundary.
 By Lemma~\ref{lm:sim:cyl:in:other:cyl}, $\gamma_i$ is contained in a simple cylinder $D_i$ whose every core curve crosses $C_i$ once. The two saddle connections in the boundary of $D_i$ are the pair $(\gamma^+_{i+1},\gamma^-_{i+1})$ in our sequence. Note that, since $\gamma^\pm_{i+1}$ do not intersect $\gamma_i$ in its interior, $\gamma^+_{i+1},\gamma^-_{i+1}$ are actually saddle connections in $(X,\omega)$.

 By construction, $\gamma^+_{i+1},\gamma^-_{i+1}, \gamma^+_i,\gamma^-_i$ bound an embedded parallelogram $P_i$ in $X$. Since $\inter(P_i)$ is contained in $\inter(C_i)$, every point in $\inter(P_i)$ is contained in a core curve of $C_i$. The same is true for every point in $\inter(\gamma^+_{i+1})$ and in $\inter(\gamma^-_{i+1})$.

 Observe that the embedded parallelograms obtained from this procedure  are pairwise disjoint  (see Figure~\ref{fig:para:dec:veech}).
 Since the total angle in a parallelogram is $2\pi$, the procedure must stop after $\vartheta(X)/(2\pi)$ iterations, where $\vartheta(X)$ is the total cone angle at the singularities of $X$, and we get a decomposition of $X$ into parallelograms.
 Since any regular point in the closure of each parallelogram in this decomposition is contained in at least one simple closed geodesic, the proposition follows.
 \end{proof}

 \begin{proof}[Proof of Theorem~\ref{thm:hyp} Case V]
 In this case, Theorem~\ref{thm:hyp} is a direct consequence of Proposition~\ref{prop:veechsf} and Theorem~\ref{thm:gen:dichotomy}.
 \end{proof}

\subsection{Case IV: $\mathcal{L}=\widetilde{\mathcal{H}}(0,0)$}\label{ssec:IV}
\begin{proof}[Proof of Theorem~\ref{thm:hyp} for Case IV]
Recall  that in this case $\omg$ must have two zeros.
We first notice that there is a map $F: \mathcal{L} \ra \mathcal{H}(0,0)$ defined as follows: given $(X,\omega)\in \mathcal{L}$, let $\Lambda_\omega:=\{\int_c \omega, \; c \in H_1(X,\Z)\}$. Then $\Lambda_\omega$ is a lattice in $\C$. Thus $\mathbb{T}:=\C/\Lambda_\omega$ is a torus. Pick a base point $p_0\in X$ and define a map $\pi: X \ra \mathbb{T}$ by $\pi(q)=\int_{p_0}^{q}\omega \mod \Lambda_\omega$ for any $q\in X$, where the integration is taken along any path from $p_0$ to $q$. It is not difficult to see that $\pi$ is well defined. Let $x_1,x_2$ be the images of the zeros of $\omega$ under $\pi$. We define $F((X,\omega))=(\mathbb{T}, \{x_1,x_2\}) \in \mathcal{H}(0,0)$.
By construction $\pi: (X,\omega) \ra (\mathbb{T},\{x_1,x_2\})$ is a translation cover which is branched over $\{x_1,x_2\}$.

\medskip

Let $p$ be a regular point on $X$.
Assume that $\pi(p) \not\in \{x_1,x_2\}$. By Lemma~\ref{thm:h0k}, there is a closed geodesic $\gamma$ through $\pi(p)$ which does not intersect the set $\{x_1,x_2\}$. Then the component of $\pi^{-1}(\gamma)$ that contains $p$ is a closed geodesic of $(X,\omega)$ (c.f. Lemma~\ref{lemma:bc}) so we are done.

\medskip

Assume now that  $\pi(p)\in\{x_1,x_2\}$, say $\pi(p)=x_1$. In this case, given a closed geodesic $\gamma \subset \mathbb{T}$ through $\pi(p)$, the component of $\pi^{-1}(\gamma)$ that contains $p$ may also contain the zero of $\omega$ that projects to $x_1$. To prove the theorem in this case, we will make use of Case V, which has been proved.

Suppose to a contradiction that $p$ is  not contained in any simple closed geodesic of $(X,\omega)$.
Recall that the set of square-tiled surfaces is dense in $\mathcal{L}$.
Let $(X',\omega')\in \mathcal{L}$ be a square-tiled surface, and $\pi':(X',\omega')\to (\mathbb{T}',\{x_1',x_2'\})$
a translation cover where  $(\mathbb{T}',\{x'_1,x'_2\})\in \mathcal{H}(0,0)$ is the image of $(X',\omega')$ under $F$.
Note that the condition $(X',\omega')$ is square-tiled means that $x'_1-x'_2$ is a torsion point of $\mathbb{T}'$.

By assumption,  there exist $a_n\in \GL(2,\R), \; n \in \N$, such that $a_n\cdot(X,\omega)\to (X',\omega')$ as $n\to\infty$.
Let $(X_n,\omega_n,p_n):=a_n\cdot(X,\omega,p)$, $(\mathbb{T}_n,\{x_{1n},x_{2n}\}):=a_n\cdot (\mathbb{T},\{x_{1},x_{2}\})$, and
$\pi_n:(X_n,\omega_n)\to(\mathbb{T}_n,\{x_{1n},x_{2n}\})$ the corresponding translation covers.
Since the map $F$ is  continuous, $(\mathbb{T}_n,\{x_{1n},x_{2n}\})$ converges to $(\mathbb{T}',\{x_1',x_2'\})$ as $n\to \infty$.

There exists some positive constant $\delta$ such that any two points in $(\pi')^{-1}(\{x_1',x_2'\}) \subset X'$ are of distance at least $\delta$ away from each other with respect to the flat metric defined by $\omega'$.
Therefore, there exists a constant $N>0$, such that for all $n>N$, $p_n$ is of distance at least $\delta/2$ away from any zero of $\omega_n$ (with respect to the flat metric defined by $\omega_n$), since $p_n\in \pi_n^{-1}(\{x_{1n},x_{2n}\})$.

After passing to a subsequence if necessary, we have $(X_n,\omega_n,p_n)$ converges to $(X',\omega',p')$ as $n\to\infty$, where $p'\in (\pi')^{-1}(\{x_1',x_2'\})$ is a regular point.
By Lemma~\ref{lm:finite:cyl:closed}, $p'$ is not contained in any simple closed geodesic of $(X',\omega')$, which contradicts Theorem~\ref{thm:hyp} for Case V, since  $(X',\omega')$ is a Veech surface. Thus Theorem~\ref{thm:hyp} holds for Case IV. The proof of Theorem~\ref{thm:hyp} is now complete.
\end{proof}


\begin{thebibliography}{amsalpha}
\bibitem{AEM12} A.~Avila, A.~Eskin, and M.~M\"{o}ller: {Symplectic and Isometric $\SL(2,\R)$-invariant subbundle of the Hodge bundle}. J. Reine  Angew. Math. 732, 1--20.

 \bibitem{AN} D. Aulicino and D.-M. Nguyen: Rank two affine submanifolds in H(2,2) and H(3,1). Geom. Topol. 20 (2016), no. 5, 2837--2904.



\bibitem{Apisa_hyp}P. Apisa: $\GL(2,\R)$ Orbit Closures in Hyperelliptic Components of Strata. Duke Math. J. 167 (2018), no. 4, 679--742.


\bibitem{Apisa_generic}P. Apisa: GL(2,R)-Invariant Measures in Marked Strata: Generic Marked Points, Earle-Kra for Strata, and Illumination.  arXiv:1601.07894v2.


\bibitem{Apisa_hyprk1} P. Apisa: Rank one orbit closures in $\mathcal H^{hyp}(g-1,g-1)$. arXiv:1710.05507v1.


\bibitem{Apisa_periodic_g2} P. Apisa: Periodic Points in Genus Two: Holomorphic Sections over Hilbert Modular Varieties, Teichmuller Dynamics, and Billiards.  arXiv:1710.05505v1.

 \bibitem{AW}P. Apisa and A. Wright: Marked points on translation surfaces.  arXiv:1708.03411v1.


 \bibitem{BGKT}M. Boshernitzan, G. Galperin, T. Kr\"uger, S. Troubetzkoy: Periodic billiard orbits are dense in rational polygons. Trans. Amer. Math. Soc.  350  (1998),  no. 9, 3523--3535.

\bibitem{Cal} K. Calta: Veech surfaces and complete periodicity in genus two.
 J. Amer. Math. Soc.  17  (2004),  no. 4, 871--908.

 \bibitem{EFW}A. Eskin, S. Filip and A. Wright: The algebraic hull of the Kontsevich-Zorich cocycle.  Ann. of Math.  188 (2018), 281--313.

 \bibitem{EM} A. Eskin and M. Mirzakhani: Invariant and stationary measures for the $\SL(2,\R)$ action on Moduli space. Publ. Math. Inst. Hautes Études Sci. 127 (2018), 95--324.

 \bibitem{EMM}A. Eskin, M. Mirzakhani and A. Mohammadi: Isolation, equidistribution, and orbit closures for the SL(2,R) action on moduli space. Ann. of Math. (2) 182 (2015), no. 2, 673--721.


\bibitem{Fil} S.~Filip: Splitting mixed Hodge structures over affine invariant manifolds, Annals of Math. {\bf 183} (2016), no.2, 681--713.

\bibitem{GHS} E. Gutkin, P. Hubert and T. Schmidt: Affine diffeomorphisms of translation surfaces: periodic points, Fuchsian groups and arithmeticity. Ann. Scient. \'Ec. Norm. Sup.(4), 36, (2003), no.6, 847--866.

\bibitem{GJ} E. Gutkin and C.Judge: Affine mappings of translation surfaces: geometry and arithmetic. Duke Math. J.  103  (2000),  no. 2, 191--213.


\bibitem{KMS}S. Kerckhoff, H. Masur and J. Smillie, Ergodicity of Billiard Flows and Quadratic Differentials. Ann. of Math.,  Vol. 124, No. 2 (1986), pp. 293--311.

\bibitem{KS00} R. Kenyon and J. Smillie: Billiards on rational-angles triangles.  Comment. Math. Helv. {\bf 75} (2000), 65--108.

\bibitem{KM} A. Kumar and R. E. Mukamel: Real multiplication through explicit correspondences. LMS J. Comput. Math. 19 (2016), suppl. A, 29--42.

\bibitem{KZ}M. Kontsevich and A. Zorich: Connected components of the moduli spaces of Abelian differentials with prescribed singularities. Invent. Math.  153  (2003),  no. 3, 631--678.



\bibitem{Mas_ergodic}H. Masur, Ergodic actions of the mapping class group. Proc. Amer. Math. Soc.  94  (1985),  no. 3, 455-–459.
\bibitem{Mas} H. Masur, Closed trajectories for quadratic differentials with an application to billiards. Duke Math. J.  53  (1986),  no. 2, 307--314.

\bibitem{MT} H. Masur and S.  Tabachnikov: Rational billiards and flat structures.  Handbook of dynamical systems, Vol. 1A,  1015--1089, North-Holland, Amsterdam, 2002.

\bibitem{McM_spin} C.T. McMullen: {Teichm\"uller} curves in genus two:  Discriminant  and  spin, Math.  Annalen  {\bf 333}  (2005),  87--130.

\bibitem{McM2}  C. T. McMullen:  Teichm\"uller curves in genus two: torsion divisors and ratios of sines. Invent. Math. 165 (2006), no. 3, 651--672.

\bibitem{McM1} C. T. McMullen: Dynamics of $\SL_2\R$ over the moduli space in genus two. Ann. of Math. 165 (2007), 397--456.

\bibitem{MW1} M. Mirzakhani and A. Wright: The boundary of an affine invariant submanifold. Invent. Math. 209 (2017), no. 3, 927--984.

\bibitem{MW2} M. Mirzakhani and A. Wright: Full rank affine invariant submanifolds. Duke Math. J.
Volume 167, Number 1 (2018), 1--40.

\bibitem{Mol1} M. M\"oller: Periodic points on Veech surfaces and the Mordell-Weil group over a Teichm\"uller curve. Invent. Math. 165 (2006), 633--649.



\bibitem{Ngu1} D.-M., Nguyen: Parallelogram decompositions and generic surfaces in $\mathcal{H}^{hyp}(4)$.
 Geom. Topol.  15  (2011),  no. 3, 1707-1747.

\bibitem{NW} D.-M. Nguyen and A. Wright: Non-Veech surfaces in $\mathcal{H}^{hyp}(4)$ are generic. Geom. Funct. Anal. 24 (2014), no. 4, 1316--1335.

\bibitem{SW1} J. Smillie and B. Weiss: Finiteness results for flat surfaces: a survey and problem list. Partially hyperbolic dynamics, laminations, and Teichm\"uller flow, 125--137, Fields Inst. Commun., 51, Amer. Math. Soc., Providence, RI, 2007.




\bibitem{Veech2} W. A. Veech: Teichm\"uller curves in moduli space, Eisenstein series and an application to triangular billiards. Invent. Math. 97 (1989), no. 3, 553--583.

\bibitem{Vor1}Y. Vorobets, Periodic geodesics on translation surfaces.  arXiv:math/0307249.

\bibitem{Vor2}Y. Vorobets, Periodic geodesics on generic translation surfaces,  Algebraic and topological dynamics, 205--258, Contemp. Math., 385, Amer. Math. Soc., Providence, RI,  2005.


\bibitem{Wright1} A. Wright: The field of definition of affine invariant submanifolds of the moduli space of abelian differentials. Geom. Topol. 18 (2014), no. 3, 1323--1341.

\bibitem{Wright2} A.~Wright: Cylinder deformations in orbit closures of translation surfaces.  Geometry $\&$ Topology {\bf 19} (2015), 413--438.



\bibitem{Zorich:survey}  A.~Zorich: Flat surfaces.   Frontiers   in  number  theory,   physics,  and  geometry, 437--583, Springer, Berlin (2006).
 \end{thebibliography}
\end{document}